\documentclass{amsart}
\usepackage{epsfig, amssymb, amsfonts, color}
\usepackage[all]{xy}
\usepackage[margin=1in]{geometry}

\newtheorem{lemma}{Lemma}[section]
\newtheorem{prop}[lemma]{Proposition}
\newtheorem{thm}[lemma]{Theorem}

\newtheorem{cor}[lemma]{Corollary}
\newtheorem{defn}[lemma]{Definition}
\newtheorem{lem}[lemma]{Lemma}
\newtheorem{rem}[lemma]{Remark}

\newtheorem*{thm*}{Theorem}

\newtheorem*{special theorem}{My Specially-Named Theorem}

\newcommand{\Z} { {\mathbb Z} }
\newcommand{\Q} { {\mathbb Q} }
\newcommand{\Spec} { {\mathrm{Spec~}} }

\newcommand{\A} { {\mathbb A} }
\newcommand{\colim} { {\mathrm{colim}} }

\newcommand{\C} { {\mathbb C} }
\newcommand{\DB} { {\mathcal D\mathcal B} }

\newcommand{\comment}[1]{}

\newcommand{\skipline}{\vspace{12pt}}

\title{Deligne-Beilinson cycle maps for Lichtenbaum cohomology}
\date{}
\author{Tohru Kohrita}
\address{Graduate School of Mathematics, Nagoya University, Furocho, Chikusaku, Nagoya 464-8602, Japan}
\email{kohrita.tohru@j.mbox.nagoya-u.ac.jp}

\begin{document}

\maketitle

\begin{abstract}
We define Deligne-Beilinson cycle maps for Lichtenbaum cohomology $H_L^m(X,\Z(n))$ and that with compact supports $H_{c,L}^m(X,\Z(n))$ of an arbitrary complex algebraic variety $X.$ When $(m,n)=(2,1),$ the homological part of our cycle map with compact supports gives a generalization of the Abel-Jacobi theorem and its projection to the Betti cohomology yields that of the Lefschetz theorem on $(1,1)$-cycles for arbitrary complex algebraic varieties. In general degrees $(m,n),$ we show that the Deligne-Beilinson cycle maps are always surjective on torsion and have torsion-free cokernels. If $m\leq 2n,$ the version with compact supports induces an isomorphism on torsion, and so does the one without compact supports if $min\{2m-1, 2\dim X+1\}\leq 2n.$ We also characterize the algebraic part of Griffiths's intermediate Jacobians with a universal property.
\end{abstract}

\section{Introduction}

Suppose $X$ is a smooth projective complex algebraic variety. The Chow group $CH^r(X)$ of cycles of codimension $r$ admits a cycle map $cl^\DB$ to the Deligne-Beilinson cohomology $H_\DB^{2r}(X,\Z(r))$ that is compatible with the Abel-Jacobi map $AJ$ and the Betti cycle map $cl^B$ (\cite[Section 7]{Esnault-Viehweg}):
\begin{equation}\label{picture to extend}
\xymatrix{0\ar[r] & CH_{hom}^r(X) \ar[r] \ar[d]_-{AJ} & CH^r(X) \ar[dr]^-{cl^B} \ar[d]_-{cl^\DB} \\
0\ar[r] & J_G^r(X) \ar[r]_-{\text{canonical}} & H_\DB^{2r}(X,\Z(r)) \ar[r] & Hdg^r(X)\ar[r] & 0} 
\end{equation}
Here, the homological part $CH_{hom}^r(X)$ is by definition the kernel of $cl^B,$ $J_G^r(X)$ is the $r$-th Griffiths intermediate Jacobian of $X$ and $Hdg^r(X)$ is the group of Hodge $r$-cycles.

With Bloch's construction of cycle maps for his higher Chow groups~\cite{Bloch cycle map}, this picture naturally extends to higher Chow groups of smooth projective complex varieties. With Rosenschon and Srinivas's cycle maps for Lichtenbaum cohomology defined in \cite{Rosenschon-Srinivas}, there is also an analogue for Lichtenbaum cohomology of smooth projective complex varieties. Moreover, as shown in there, Bloch's cycle maps for higher Chow groups factor through Rosenschon-Srinivas's cycle maps for Lichtenbaum cohomology.

The purpose of this article is to extend the picture~(\ref{picture to extend}) to Lichtenbaum cohomology of arbitrary complex varieties $X,$ and study, in this \'etale context, Abel-Jacobi maps and Betti cycle maps through Deligne-Beilinson cycle maps. 

Our formulation uses eh hypercohomology of \cite{Geisser eh}. By Lichtenbaum cohomology $H_{(c),L}^m(X,\Z(n))$ (resp., with compact supports), we mean the eh hypercohomology (resp., with compact supports) of the eh sheafification of the Suslin-Friedlander motivic complex $\Z(n)^{SF}$ (see \cite[Lecture 16]{MVW}). The Deligne-Beilinson cohomology $H_{(c),\DB}^m(X,\Z(n))$ (resp., with compact supports) means eh hypercohomology (resp., with compact supports) of the eh sheafification of a certain complex of Zariski sheaves defined in \cite[(1.6.5)]{Beilinson} (see Definition~\ref{defn: cdh definition of DB}).

We construct the following analogue of the diagram~(\ref{picture to extend}) for an arbitrary complex algebraic variety $X$ (the diagram~(\ref{for corollary})):
\begin{displaymath}
\xymatrix{ 0\ar[r] & H_{(c),L,hom}^m(X,\Z(n)) \ar[r] \ar[d]_-{AJ_{(c),L}^{m,n}} & H_{(c),L}^m(X,\Z(n)) \ar[d]_-{cl_{(c),L}^\DB} \ar[r]^-{cl_{(c),L}^B} & im(cl_{(c),L}^B) \ar[r] \ar@{_{(}->}[d]_-{\text{inc.}} & 0 \\
0\ar[r] & J_{(c)}^{m,n}(X) \ar[r] &  H_{(c),\DB}^m(X,\Z(n)) \ar[r] & H_{(c),B}^m(X,\Z(n))\cap F^nH_{(c),B}^m(X,\C) \ar[r] &0.}
\end{displaymath}
Here, $F^nH_{(c),B}^m(X,\C)$ signifies the $n$-th Hodge filtration. $cl_{(c),L}^\DB: H_{(c),L}^m(X,\Z(n)) \longrightarrow H_{(c),\DB}^m(X,\Z(n))$ is the Deligne-Beilinson cycle map (with compact support) in Definition~\ref{defn: DB cycle map with compact supports}. $H_{(c),L,hom}^m(X,\Z(n))$ is the kernel of the Betti cycle map $cl_{(c),L}^B: H_{(c),L}^m(X,\Z(n)) \longrightarrow H_{(c),B}^m(X,\Z(n))$ whose definition is analogous to that of $cl_{(c),L}^\DB.$ $AJ_{(c),L}^{m,n}$ is defined as the restriction of $cl_{(c),L}^\DB,$ and its target $J_{(c)}^{m,n}(X)$ is the Carlson $n$-th intermediate Jacobian associated with the mixed Hodge structure of $H_{c,B}^{m-1}(X,\Z(n))$ (\cite{Carlson}). The homological part $H_{(c),L,hom}^m(X,\Z(n))$ agrees with the subgroup of all divisible elements of $H_{(c),L}^m(X,\Z(n))$ (Remark~\ref{rem: max div}). 

The cycle maps $cl_{(c),L}^\DB$ and $cl_{(c),L}^B$ are constructed in the derived category of sheaves on the eh site $Sch/\C_{eh}.$ If $X$ is smooth and proper, our cycle maps agree with the usual cycle maps for Chow groups after the composition with the canonical map $CH^r(X)\longrightarrow H_{L}^{2r}(X,\Z(r)).$

With the above diagram for $(m,n)=(2,1),$ we can generalize the Abel-Jacobi theorem and the Lefschetz theorem on $(1,1)$-cycles to singular varieties.

\begin{thm*}[{Theorem~\ref{thm: AJ and Lefschetz} and Remark~\ref{rem: max div}; cf. \cite[Theorem 7.8]{Arapura}}]
Suppose $X$ is an arbitrary separated scheme $X$ of finite type over $\C.$ Then, there are isomorphisms
\begin{displaymath}
\xymatrixcolsep{3.5pc}\xymatrix{H_{c,cdh}^2(X,\Z(1)_{cdh}^{SF})_{div} \ar[r]^-\cong_-{\text{canonical}} & H_{c,L,hom}^2(X,\Z(1)) \ar[r]^-\cong_-{AJ_{c,L}^{2,1}} & J_c^{2,1}(X)}
\end{displaymath}
and a surjection
\begin{displaymath}
\xymatrixcolsep{3.5pc}\xymatrix{H_{c,cdh}^2(X,\Z(1)_{cdh}^{SF}) \ar[r]^-\cong_-{\text{canonical}} & H_{c,L}^2(X,\Z(1)) \ar@{->>}[r]^-{cl_{c,L}^{B}} & H_{c,B}^2(X,\Z(1))\cap F^1H_{c,B}^2(X,\C),}
\end{displaymath}
where $H_{c,cdh}^2(X,\Z(1)_{cdh}^{SF})_{div}$ is the maximal divisible subgroup of $H_{c,cdh}^2(X,\Z(1)_{cdh}^{SF}).$
\end{thm*}

For other degrees $(m,n),$ we show that the cycle maps $cl_L^\DB$ and $cl_{c,L}^\DB$ are always surjective on torsion and have torsion-free cokernels (Theorem~\ref{thm: torsion main}~(i)). We shall also prove the following.

\begin{thm*}[{Theorem~\ref{thm: torsion main} and Corollary~\ref{cor: on torsion}; cf. \cite[Proposition 5.1]{Rosenschon-Srinivas} for the smooth projective case}]
For an arbitrary separated scheme of finite type over $\C$ and non-negative integers $m$ and $n,$
\begin{enumerate}
\item If $m\leq 2n,$ then the maps $cl_{c,L}^\DB$ and $AJ_{c,L}^{m,n}$ are isomorphisms on torsion.
\item If $min\{2m-1,2\dim X+1\}\leq 2n,$ then the maps $cl_L^\DB$ and $AJ_L^{m,n}$ are isomorphisms on torsion.
\item If $m< 2n,$ then $im(cl_{c,L}^B)=H_{c,B}^m(X,\Z(n))\cap F^nH_{c,B}^m(X,\C)=H_{c,B}^m(X,\Z(n))_{tor}.$
\item If $min\{2m, 2\dim X\}<2n,$ then $im(cl_{L}^B)=H_{B}^m(X,\Z(n))\cap F^nH_{B}^m(X,\C)=H_{B}^m(X,\Z(n))_{tor}.$
\end{enumerate}
\end{thm*}

With this theorem in hand, we give an algebraic construction of the ``algebraic part" of Griffiths's intermediate Jacobians by a universal property. This construction works over any algebraically closed field.

Here is an outline of the paper. In Section~\ref{section: Preliminaries: Deligne-Beilinson cohomology}, we give a detailed proof of the \'etale descent for Deligne-Beilinson cohomology and other cohomology theories of our interest. In Section~\ref{section: cdh method}, we prove the eh descent for these cohomologies. This enables us to define Deligne-Beilinson cohomology (with compact supports) for arbitrary complex algebraic varieties. After checking that our generalization behaves well with respect to mixed Hodge structures in a certain sense in Section~\ref{section: Intermediate Jacobians with compact supports}, we proceed to define Deligne-Beilinson cycle maps for Lichtenbaum cohomology (with compact supports) of arbitrary complex varieties in Section~\ref{section: Deligne-Beilinson cycle maps}. At this point, our versions of the Abel-Jacobi theorem and the Lefschetz theorem for singular varieties follow immediately. In Section~\ref{section: Study of torsion elements}, we study the torsion part of the Deligne-Beilinson cycle maps, and as a corollary, we characterize the ``algebraic part" of Griffiths intermediate Jacobians by a universal property in Section~\ref{section: Griffiths intermediate Jacobians revisited}.

\skipline
\emph{Convention.}
Schemes are assumed to be separated and of finite type over the field $\C$ of complex numbers. The category of all schemes (resp., all smooth schemes) is denoted by $Sch/\C$ (resp., $Sm/\C$). 

For an arbitrary scheme $X$ and integers $m$ and $n$ ($n\geq0$), Lichtenbaum cohomology and Lichtenbaum cohomology with compact supports mean:
\begin{itemize}
\item {\bf motivic cohomology}: 
$H_L^m(X,\Z(n)):=H_{eh}^m(X,\Z(n)_{eh}^{SF}),$
\item {\bf motivic cohomology with compact supports}: 
$H_{c,L}^m(X,\Z(n)):=H_{c,eh}^m(X,\Z(n)_{eh}^{SF}),$
\end{itemize}
where the right hand sides are eh hypercohomology (with compact supports) of the eh sheafification of the Suslin-Friedlander motivic complex $\Z(n)^{SF}.$ See \cite[Section 3]{Geisser eh} for eh hypercohomology with compact supports and \cite[Lecture 16]{MVW} for $\Z(n)^{SF}.$

\skipline
\emph{Acknowledgements.}
The author would like to thank Thomas Geisser and Shane Kelly for helpful conversations and suggestions.


\section{Preliminaries: Deligne-Beilinson cohomology}\label{section: Preliminaries: Deligne-Beilinson cohomology}

We give a detailed proof of the \'etale descent property of Deligne-Beilinson cohomology while organizing relevant topics (contained in \cite{Beilinson, Esnault-Viehweg}) for our purpose. We call a smooth compactification of a smooth scheme a \emph{good compactification} if the boundary divisor is strict normal crossing.

\begin{defn}[\cite{Beilinson, Esnault-Viehweg}.]\label{defn: DB cohomology for smooth schemes}
Let $X$ be a smooth scheme (over $\C$), and let $X\buildrel j\over\hookrightarrow \bar X$ be a good compactification with the boundary divisor $Z.$ The {\bf Deligne-Beilinson cohomology} of $X$ is defined as analytic hypercohomology
$$H_\DB^m(X,\Z(n)):=H_{an}^m(\bar X,\Z(n)_{\bar X,Z}^\DB),$$
where
$$\Z(n)_{\bar X,Z}^\DB:=cone(Rj_*\underline{\Z(n)}\oplus\Omega_{\bar X}^{\bullet\geq n}(\log Z)\buildrel\epsilon-\iota\over\longrightarrow Rj_*\Omega_X^\bullet)[-1]$$
is a complex of analytic sheaves on $\bar X.$ Here, $\underline{\Z(n)}$ is the constant sheaf of value $(2\pi i)^n\Z,$ $\Omega_X^\bullet$ is the de Rham complex of holomorphic forms on $X$ and $\Omega_{\bar X}^{\bullet\geq n}(\log Z)$ is the brutal truncation of the complex of meromorphic differential forms on $\bar X$ with at most logarithmic poles along $Z.$ The maps $\epsilon$ and $\iota$ are the canonical ones.
\end{defn}

As the notation suggests, Definition~\ref{defn: DB cohomology for smooth schemes} is independent of the choice of the good compactification $\bar X$ (\cite[Lemma 2.8]{Esnault-Viehweg}). We may actually express the Deligne-Beilinson cohomology of a smooth scheme $X$ as hypercohomology of $X$ itself with coefficients in a complex of sheaves on the big Zariski site $Sm/\C_{Zar}.$ Here is Beilinson's construction of such a complex.

Consider the category $\Pi$ of good compactifications whose objects are good compactifications $j: T\hookrightarrow\bar T$ ($T$ is any smooth scheme) and morphisms are commutative diagrams
\begin{displaymath}
\xymatrix{ T \ar@{^{(}->}[r]^j \ar[d] & \bar T \ar[d]\\
 T' \ar@{^{(}->}[r]^{j'}  & \bar T' }
\end{displaymath}
We would like to define a suitable notion of analytic sheaves on $\Pi.$ 

We start by defining an analytic sheaf on a good compactification $j: T\hookrightarrow \bar T$ as a pair of analytic sheaves $\mathcal F$ on $T$ and $\bar{\mathcal F}$ on $\bar T$ together with a sheaf morphism $\phi: \bar{\mathcal F}\longrightarrow j_*\mathcal F.$ The analytic sheaves on good compactification $j: T\hookrightarrow \bar T$ form the category $Sh_{an}(\bar T,T)$ with morphisms pairs $(\bar f:\bar{\mathcal F}\longrightarrow \bar{\mathcal F'},f: \mathcal F\longrightarrow \mathcal F')$ of maps of analytic sheaves that make the diagram
\begin{displaymath}
\xymatrix{ \bar{\mathcal F} \ar[r]^\phi \ar[d]_{\bar f} & j_*\mathcal F \ar[d]^{j_*f}\\
\bar{\mathcal F'} \ar[r]^{\phi'} & j_*\mathcal F'}
\end{displaymath}
commute. 

The category $Sh_{an}(\bar T,T)$ has enough injectives. Indeed, if $\bar i:\bar{\mathcal F}\hookrightarrow \bar{\mathcal I}$ and $i:\mathcal F\hookrightarrow \mathcal I$ are imbeddings into injective sheaves, the map 
\begin{equation}\label{injective imbedding}
(\bar i\oplus (j_*(i)\circ\phi),i):(\bar{\mathcal F},\mathcal F,\phi)\longrightarrow (\bar{\mathcal I}\oplus j_*\mathcal I,\mathcal I, pr_2)
\end{equation}
is an imbedding into an injective object in $Sh_{an}(\bar T,T)$ (\cite[4.2]{Esnault-Viehweg}). In particular, if $\bar{\mathcal F}\longrightarrow(\bar{\mathcal I}^\bullet,\bar d)$ and $\mathcal F\longrightarrow(\mathcal I^\bullet,d)$ are injective resolutions, then 
$$(\bar{\mathcal F}, 0,0)\longrightarrow (\bar{\mathcal I}^\bullet,0,0)$$ and 
$$(0,\mathcal F,0)\longrightarrow (j_*\mathcal I^0,\mathcal I^0,id)\buildrel{\delta^0}\over\longrightarrow(j_*\mathcal I^0\oplus j_*\mathcal I^1,\mathcal I^1,pr_2)\buildrel{\delta^1}\over\longrightarrow (j_*\mathcal I^1\oplus j_*\mathcal I^2,\mathcal I^2,pr_2)\buildrel{\delta^2}\over\longrightarrow (j_*\mathcal I^2\oplus j_*\mathcal I^3,\mathcal I^3,pr_2)\longrightarrow\cdots$$
are injective resolutions in $Sh_{an}(\bar T,T).$ Here, the differential $\delta^n$ sends 
$$(x\oplus y,z)\in (j_*\mathcal I^{n-1}\oplus j_*\mathcal I^n,\mathcal I^n,pr_2)$$ to 
$$(-j_*(d)(x)+y\oplus j_*(d)(y),d(z))\in (j_*\mathcal I^{n}\oplus j_*\mathcal I^{n+1},\mathcal I^{n+1},pr_2).$$
Note that these descriptions are still valid even when $\bar{\mathcal F}$ and $\mathcal F$ are replaced with complexes of sheaves $\bar{\mathcal F}^\bullet$ and $\mathcal F^\bullet$ and $\bar{\mathcal I}^\bullet$ and $\mathcal I^\bullet$ with their injective resolutions; i.e. 
\begin{equation}\label{explicit resolution first}
(\bar{\mathcal F}^\bullet,0,0)\longrightarrow (\bar{\mathcal I}^\bullet,0,0)
\end{equation}
and 
\begin{equation}\label{explicit resolution}
(0,\mathcal F^\bullet,0)\longrightarrow (j_*\mathcal I^\bullet[-1]\oplus j_*\mathcal I^\bullet,\mathcal I^\bullet,pr_2) 
\end{equation}
are quasi-isomorphisms. 

A collection of analytic sheaves $\{(\bar{\mathcal F}_{\bar T},\mathcal F_T,\phi_j)\}_{j:T\hookrightarrow\bar T}$ on $j:T\hookrightarrow\bar T$ together with a morphism $f^\sharp:(\bar{\mathcal F}_{\bar T'},\mathcal F_{T'},\phi_{j'})\longrightarrow (\bar f_*\bar{\mathcal F}_{\bar T},f_*\mathcal F_T,f_*\phi_j)$ for each morphism $(f,\bar f):(T\buildrel j\over\hookrightarrow\bar T)\longrightarrow (T'\buildrel j'\over\hookrightarrow\bar T')$ in $\Pi$ such that $(f\circ g)^\sharp=g^\sharp\circ f^\sharp$ and $id^\sharp=id$ is called an analytic sheaf on $\Pi.$ The category of analytic sheaves on $\Pi$ is written as $Sh_{an}(\Pi).$ Since analytic sheaves have functorial Godement injective imbeddings, the imbedding~(\ref{injective imbedding}) can be chosen functorially with respect to the morphisms of $\Pi.$ Thus, $Sh_{an}(\Pi)$ has enough injectives.

Now, let $\tau\in\{triv, Zar, Nis, \acute et\}$ ($triv$ stands for the trivial topology) and consider the functor $\sigma_\tau: Sh_{an}(\Pi)\longrightarrow Sh_{\tau}(Sm/\C)$ that sends $\{(\bar{\mathcal F}_{\bar T}, \mathcal F_T,\phi_j)\}_{j:T\hookrightarrow \bar T}$ to the $\tau$-sheafification of the presheaf
$$Sm/\C\ni T\mapsto \colim_{j:T\hookrightarrow \bar T}\Gamma_{\bar T,T}(\bar{\mathcal F}_{\bar T}, \mathcal F_T,\phi_j),$$
where the colimit is taken over all good compactifications of $T$ and $\Gamma_{\bar T,T}$ is the left exact functor $\Gamma_{\bar T,T}(\bar{\mathcal F}_{\bar T}, \mathcal F_T,\phi_j)=\mathrm{ker}\{ \bar{\mathcal F}_{\bar T}(\bar T) \buildrel\phi_j\over\longrightarrow\mathcal F_T(T)\}.$ Since the colimit is taken over a directed system and sheafification is an exact functor, $\sigma_\tau$ is left exact. Deriving it, we obtain
$$R\sigma_\tau: D^+(Sh_{an}(\Pi))\longrightarrow D^+(Sh_{\tau}(Sm/\C)).$$

It is important for us that, given a complex of sheaves $\{(\bar{\mathcal F}_{\bar T}^\bullet, \mathcal F_T^\bullet,\phi_j^\bullet)\}_{j:T\hookrightarrow \bar T}$ in $Sh_{an}(\Pi),$ there is always a canonical map of presheaves
\begin{equation}\label{presheaf cohomology to sheaf cohomology}
R\sigma_{triv}(\{(\bar{\mathcal F}_{\bar T}^{\bullet},\mathcal F_T^\bullet,\phi_{j}^\bullet)\}_{j:T\hookrightarrow\bar T})\buildrel\text{sheafification}\over\longrightarrow R\sigma_\tau(\{(\bar{\mathcal F}_{\bar T}^{\bullet},\mathcal F_T^\bullet,\phi_{j}^\bullet)\}_{j:T\hookrightarrow\bar T}).
\end{equation}
For each $X\in Sm/\C,$ it induces a homomorphism
\begin{equation}\label{canonical map}
H^m(R\sigma_{triv}(\{(\bar{\mathcal F}_{\bar T}^{\bullet},\mathcal F_T^\bullet,\phi_{j}^\bullet)\}_{j:T\hookrightarrow\bar T})(X))\longrightarrow H_\tau^m(X, R\sigma_\tau(\{(\bar{\mathcal F}_{\bar T}^{\bullet},\mathcal F_T^\bullet,\phi_{j}^\bullet)\}_{j:T\hookrightarrow\bar T})).
\end{equation}
Note that, by definition, the question whether the map~(\ref{canonical map}) is an isomorphism for all $m$ is a question if the restriction of the complex of presheaves $R\sigma_{triv}(\{(\bar{\mathcal F}_{\bar T}^{\bullet},\mathcal F_T^\bullet,\phi_{j}^\bullet)\}_{j:T\hookrightarrow\bar T})$ to the small site $X_\tau$ satisfies $\tau$-descent.

It is convenient to have the following lemma.

\begin{lem}[{cf. \cite[Proposition 4.4 (a)]{Esnault-Viehweg}}]\label{lem: triangle EV4.4}
For any complex $\{(\bar{\mathcal F}_{\bar T}^\bullet,\mathcal F_T^\bullet,\phi_{j}^\bullet)\}_{j:T\hookrightarrow\bar T}$ of analytic sheaves on $\Pi,$ we have
$$R\sigma_\tau(\{(\bar{\mathcal F}_{\bar T}^\bullet,\mathcal F_T^\bullet,\phi_{j}^\bullet)\}_{j:T\hookrightarrow\bar T})=cone\big(R\sigma_\tau(\{(\bar{\mathcal F}_{\bar T}^\bullet,0,0)\}_{j:T\hookrightarrow\bar T})\buildrel R\sigma_\tau(\phi_{j}^\bullet)\over\longrightarrow R\sigma(\{(0,\mathcal F_T^\bullet[1],0)\}_{j:T\hookrightarrow\bar T})\big)[-1].$$
\end{lem}

\begin{proof}
It is enough to observe the identity for injective sheaves of the form $\{(\bar{\mathcal I}_{\bar T}\oplus j_*\mathcal I_T,\mathcal I_T, pr_2)\}_{j:T\hookrightarrow \bar T},$ where $\bar{\mathcal I}_{\bar T}$ (resp., $\mathcal I_T$) is an injective sheaf on $\bar T$ (resp., $T$) because every sheaf in $Sh_{an}(\Pi)$ can be imbedded into an injective sheaf of this form.
\end{proof}

Here is a list of complexes of sheaves of our concern in this article.
\begin{enumerate}
\item $\{DB(n)_{\bar T, T}\}_{T\hookrightarrow \bar T}:=\{(\Omega_{\bar T}^{\bullet\geq n}(\log (\bar T\setminus T)), cone(\underline{\Z(n)}\buildrel\epsilon\over\longrightarrow \Omega_T^\bullet),-\iota)\}_{T\hookrightarrow \bar T}$
\item $\{DR_{\bar T, T}\}_{T\hookrightarrow \bar T}:=\{(0, \Omega_T^\bullet[1],0)\}_{T\hookrightarrow \bar T}$
\item $\{DR_{\bar T, T}^{\geq n}\}_{T\hookrightarrow \bar T}:=\{(\Omega_{\bar T}^{\bullet\geq n}(\log(\bar T\setminus T)), 0,0)\}_{T\hookrightarrow \bar T}$
\item $\{DR_{\bar T, T}^{<n}\}_{T\hookrightarrow \bar T}:=\{(\Omega_{\bar T}^{\bullet\geq n}(\log(\bar T\setminus T))[1], \Omega_T^\bullet[1],-\iota)\}_{T\hookrightarrow \bar T}$
\item $\{B(n)_{\bar T, T}\}_{T\hookrightarrow \bar T}:=\{(0, \underline{\Z(n)}[1],0)\}_{T\hookrightarrow \bar T}$
\item $\{B(\Z/a)_{\bar T,T}\}_{T\hookrightarrow \bar T}:=\{(0,\underline{\Z/a}[1],0)\}_{T\hookrightarrow \bar T}~~~\text{($a$ is an arbitrary positive integer)}.$
\end{enumerate}
In (v) and (vi), $\underline{\Z(n)}:=\underline{2n\pi\Z}$ and $\underline{\Z/a}$ denote the constant analytic sheaves. We write the respective images of these complexes under the functor $R\sigma_\tau: D^+(Sh_{an}(\Pi))\longrightarrow D^+(Sh_{\tau}(Sm/\C))$ as:
\begin{enumerate}
\item $DB(n)_{\tau}$
\item $DR_{\tau}$
\item $DR_{\tau}^{\geq n}$
\item $DR_{\tau}^{<n}$
\item $B(n)_{\tau}$
\item $B(\Z/a)_{\tau}.$ 
\end{enumerate}

\begin{prop}\label{prop: triangles}
In $D^+(Sh_{\tau}(Sm/\C)),$ there are distinguished triangles:
\begin{equation}\label{triangle hodge}
DR_{\tau}^{\geq n}\longrightarrow DR_{\tau}\longrightarrow DR_{\tau}^{<n}\buildrel[+1]\over\longrightarrow,
\end{equation}
\begin{equation}\label{triangle DB}
DR_{\tau}^{<n}[-1]\longrightarrow DB(n)_{\tau}\longrightarrow B(n)_{\tau}\buildrel [+1]\over\longrightarrow, 
\end{equation}
\begin{equation}\label{triangle DB multiplication}
DB(n)_{\tau}\buildrel \times a\over\longrightarrow DB(n)_{\tau}\longrightarrow B(\Z/a)_{\tau}\buildrel [+1]\over\longrightarrow,
\end{equation}
where $a$ is a positive integer.
\end{prop}

\begin{proof}
It is immediate from Lemma~\ref{lem: triangle EV4.4}.
\end{proof}

\begin{thm}[Etale descent; cf.~{\cite[1.6.5]{Beilinson},~\cite[5.5]{Esnault-Viehweg},~\cite[Theorem 2.8(vi)]{Holmstrom-Scholbach}}]\label{thm: etale descent}
The complexes of presheaves $DB(n)_{triv}, DR_{triv}, DR_{tirv}^{\geq n}, DR^{<n}_{triv}, B(n)_{triv},$ and $B(\Z/a)_{triv}$ satisfy etale descent on any $X\in Sm/\C;$ and for $\tau\in\{Zar, Nis, \acute et\},$ the map~(\ref{canonical map}) is identified respectively with
\begin{enumerate}
\item $H_{\tau}^m(X,DB(n)_{\tau})\buildrel\cong\over\longleftarrow H_\DB^m(X,\Z(n)),$
\item $H_{\tau}^m(X,DR_{\tau})\buildrel\cong\over\longleftarrow H_{dR}^m(X,\C),$
\item $H_{\tau}^m(X,DR_{\tau}^{\geq n})\buildrel\cong\over\longleftarrow H_{an}^m(\bar X,\Omega_{\bar X}^{\bullet\geq n}(\log \bar X\setminus X)),$
\item $H_{\tau}^m(X,DR^{<n}_{\tau})\buildrel\cong\over\longleftarrow H_{dR}^m(X,\C)/F^nH_{dR}^m(X,\C),$
\item $H_{\tau}^m(X,B(n)_{\tau})\buildrel\cong\over\longleftarrow H_{B}^m(X,\Z(n)),$
\item $H_{\tau}^m(X,B(\Z/a)_{\tau})\buildrel\cong\over\longleftarrow H_{B}^m(X,\Z/a).$
\end{enumerate}
In {\rm (iv)}, $F^n$ signifies the $n$-th Hodge filtration.
\end{thm}

\begin{proof}
We proceed as in the proof of \cite[Theorem 2.8(vi)]{Holmstrom-Scholbach}. Let us start with (ii), (v) and (vi). Let $cl$ denote the topology on $Sm/\C$ generated by surjective families of local isomorphisms. It is a finer topology than any of the $\tau$-topologies \cite[(4.0)]{SGA Artin}. Suppose $\mathcal F^\bullet$ is a complex of sheaves on $(Sm/\C)_{an}$ and we write its restriction to $T\in Sm/\C$ as $\mathcal F_T^\bullet.$ Consider the complex $\{(0,\mathcal F_T^\bullet,0)\}_{T\hookrightarrow\bar T}$ of analytic sheaves on $\Pi.$ Suppose $\mathcal I^\bullet$ is an injective resolution of $\mathcal F^\bullet$ in $Sh_{cl}(Sm/\C).$ Since an injective resolution in $Sh_{cl}(\Pi)$ is still an injective resolution in $Sh_{an}(\Pi)$ (note that any $cl$-cover has a refinement by an analytic cover) by the explicit description~(\ref{explicit resolution}) of injective resolutions in $Sh_{an}(\Pi),$ we have 
$$R\sigma_{\tau}(\{(0,\mathcal F_T^\bullet,0)\}_{T\hookrightarrow \bar T})=\mathcal I^\bullet[-1],$$
where $\mathcal I^\bullet$ on the right hand side is regarded as a complex of $\tau$-sheaves.

The isomorphisms (ii), (v) and (vi) are now clear. For example, for (ii), choose an injective resolution $\Omega^\bullet\longrightarrow \mathcal I^\bullet$ in $Sh_{cl}(Sm/\C).$ Since $\mathcal I^\bullet$ is still a complex of injective sheaves when restricted to the $\tau$-site, both left and right sides of (ii) is the cohomology of $\mathcal I^\bullet(X).$

For the remaining (i), (iii) and (iv), (i) and (iv) follow from (iii) and the already proven (ii), (v) and (vi). Let us explain how (iv) follows. For any smooth scheme $X,$ the distinguished triangle~(\ref{triangle hodge}) in $D^+(Sh_\tau(Sm/\C))$ and the map~(\ref{canonical map}) give rise to the commutative diagram with exact rows
\begin{displaymath}
\xymatrix{\cdots\ar[r] & H^m(DR_{triv}^{\geq n}(X)) \ar[r] \ar[d]_{f_1} & H^m(DR_{triv}(X)) \ar[r] \ar[d]_{f_2} & H^m(DR_{triv}^{<n}(X)) \ar[d]_{f_3}\\
\cdots\ar[r] & H_\tau^m(X,DR_\tau^{\geq n}) \ar[r] & H_\tau^m(X,DR_\tau) \ar[r] & H_\tau^m(X,DR_\tau^{<n})}
\end{displaymath}
\begin{displaymath}
\xymatrix{\ar[r] & H^{m+1}(DR_{triv}^{\geq n}(X)) \ar[r] \ar[d]_{f_4} & H^{m+1}(DR_{triv}(X)) \ar[r] \ar[d]_{f_5} & \cdots \\
\ar[r] & H_\tau^{m+1}(X,DR_\tau^{\geq n}) \ar[r] & H_\tau^{m+1}(X,DR_\tau) \ar[r] & \cdots}
\end{displaymath}
Let us note that 
$$H^m(DR_{triv}^{\geq n}(X))=R^m(\colim\Gamma_{\bar X,X})(\{DR_{\bar T, T}^{\geq n}\}_{T\hookrightarrow \bar T})\cong \colim H_{an}^m(\bar X,\Omega_{\bar X}^{\bullet\geq n}(\log \bar X\setminus X))\cong H_{an}^m(\bar X,\Omega_{\bar X}^{\bullet\geq n}(\log \bar X\setminus X))$$
by the independence of the Hodge filtration from the choice of a good compactification (\cite[Th\'eor\`me 3.2.5]{DeligneII}). Moreover, since the canonical map $H_{an}^i(\bar X,\Omega_{\bar X}^{\bullet\geq n}(\log\bar X\setminus X))\longrightarrow H_{an}^i(\bar X,\Omega_{\bar X}^\bullet(\log\bar X\setminus X))$ is injective for all $i$ (\cite[Scholie 8.1.9~(v)]{Deligne}), we have isomorphisms 
$$H^m(DR_{triv}^{<n}(X))\cong  H^m(DR_{triv}(X))/H^m(DR_{triv}^{\geq n}(X))    \cong H_{dR}^m(X,\C)/F^nH_{dR}^m(X,\C).$$

Now, (iv) follows from (ii) and (iii) by the 5-lemma applied to the above diagram. The isomorphism (i) can also be obtained from (iv) and (v) by a similar argument with the triangle~(\ref{triangle DB}).  
  
Let us now prove (iii), or that the complex $DR_{triv}^{\geq n}$ of presheaves, when restricted to the small $\tau$-site on each $X\in Sm/\C,$ satisfies $\tau$-descent. Recall the following criteria:
\begin{itemize}
\item Zariski descent holds if for every $U, V\in X_{Zar},$ the square
\begin{displaymath}
\xymatrix{ DR_{triv}^{\geq n}(U\cup V) \ar[r] \ar[d] & DR_{triv}^{\geq n}(U) \ar[d] \\
DR_{triv}^{\geq n}(V) \ar[r] & DR_{triv}^{\geq n}(U\cap V)}
\end{displaymath}
is homotopy cartesian (\cite{Brown-Gersten}).
\item Nisnevich descent holds if, for any Nisnevich distinguished square 
\begin{displaymath}
\xymatrix{  V \ar[r] \ar[d] & S \ar[d]^f \\
U \ar[r]_j & T}
\end{displaymath}
in $X_{Nis}$ (i.e. $f$ is \'etale, $j$ is an open immersion, and $f$ induces an isomorphism $(Y\setminus V)_{red}\buildrel\cong\over\longrightarrow (X\setminus U)_{red}$), then the induced square
\begin{displaymath}
\xymatrix{ DR_{triv}^{\geq n}(T) \ar[r] \ar[d] & DR_{triv}^{\geq n}(U) \ar[d] \\
DR_{triv}^{\geq n}(S) \ar[r] & DR_{triv}^{\geq n}(V)}
\end{displaymath}
is homotopy cartesian (\cite[Theorem 3.3.2]{CD}).
\item Etale descent holds if in addition to the condition for Nisnevich descent, given any finite Galois cover $V\longrightarrow U$ with a group $G$ in $X_{\acute et},$ the induced chain map $DR_{triv}^{\geq n}(U)\longrightarrow DR_{triv}^{\geq n}(V)^G$ to the $G$-invariant part of $DR_{triv}^{\geq n}(V)$  is a quasi-isomorphism ([Ibid.,Theorem 3.3.23]). (Theorem 3.3.23 cited is valid with rational coefficients, but this is not a problem for us because, in constructing $DR_{triv}^{\geq n},$ we can use the Godement injective resolution of $\Omega_{\bar U}^\bullet(\log \bar U\setminus U),$ which is of rational coefficients.)
\end{itemize}

Obviously, we only need to check the conditions for Nisnevich distinguished squares and finite Galois covers. Since the cohomology of the complex $DR_{triv}^{\geq n}(S)$ of global sections over any scheme $S\in Sm/\C$ is nothing but $\colim_{S\buildrel{\text{good}}\over\hookrightarrow \bar S}H_{an}^*(S,\Omega_{\bar S}^{\bullet\geq n}(\log\bar S\setminus S))$ and this group canonically injects into $H_{dR}^*(S,\C)$ (\cite[Scholie 8.1.9~(v)]{Deligne}) with the image $F^n H_{dR}^*(S,\C),$ the condition for Nisnevich distinguished squares boils down to the exactness of the complex
$$\cdots\longrightarrow F^n H_{dR}^{m-1}(V,\C)\longrightarrow F^n H_{dR}^m(T,\C)\longrightarrow F^n H_{dR}^m(S,\C)\oplus F^n H_{dR}^m(U,\C)\longrightarrow F^n H_{dR}^{m}(V,\C)\longrightarrow\cdots $$
and the condition for finite Galois covers to the isomorphism
$$F^nH_{dR}^m(U,\C)\longrightarrow F^nH_{dR}^m(V,\C)^G.$$
But, both of these follow from the strict compatibility of morphisms of mixed Hodge structures (\cite[Th\'eor\`eme 2.3.5~(iii)]{DeligneII}) because we know, for example by (ii), that the complex 
$$\cdots\longrightarrow H_{dR}^{m-1}(V,\C)\longrightarrow H_{dR}^m(T,\C)\longrightarrow H_{dR}^m(S,\C)\oplus H_{dR}^m(U,\C)\longrightarrow F^n H_{dR}^{m}(V,\C)\longrightarrow\cdots $$
is exact and the map
$$H_{dR}^m(U,\C)\longrightarrow H_{dR}^m(V,\C)^G$$
is an isomorphism.
\end{proof}


\section{$eh$ descent for Deligne-Beilinson cohomology}\label{section: cdh method}

We extend the definition of Deligne-Beilinson cohomology $H_{\DB}^m(X,\Z(n))$ (Definition~\ref{defn: DB cohomology for smooth schemes}) to arbitrary (complex) schemes and also define its compactly supported version. This is done by interpreting Deligne-Beilinson cohomology as eh hypercohomology. We shall see that our generalization behaves reasonably with respect to the Hodge structure of Betti cohomology (see Proposition~\ref{prop: reasonable} and Section~\ref{section: Intermediate Jacobians with compact supports}).

Let $\tau\in\{cdh, eh\}.$ By cdh or eh sheafification, we always mean the composition of the functors
$$Sh_{Zar}(Sm/k)\buildrel\text{$t$-sheafification}\over\longrightarrow Sh_t(Sm/k)\longrightarrow Sh_{\tau}(Sch/k),$$
where $t$ signifies the Grothendieck topology obtained by restricting the $\tau$-topology on $Sch/k$ to $Sm/k.$ The second functor is the left adjoint to the forgetful functor $Sh_{\tau}(Sch/k)\longrightarrow Sh_t(Sm/k);$ it is an equivalence of categories because resolution of singularities assures that any scheme has a smooth $\tau$-cover (\cite[proof of Lemma 3.6]{Friedlander-Voevodsky}). 

The argument of \cite[Theorem 3.6]{Geisser eh} proves the following theorem. See also \cite[Theorem 3.3.8]{CD}.

\begin{thm}\label{thm: cdh descent for DB}
Let $\mathcal F$ be a bounded below complex of \'etale sheaves on $Sm/\C.$ Suppose that for any abstract blow-up square
\begin{displaymath}
\xymatrix{ Z' \ar[r]^{i'} \ar[d]_{f'} & X' \ar[d]^f\\ 
Z \ar[r]_i & X}
\end{displaymath}
in $Sm/\C,$ there is a long exact sequence
$$\cdots\longrightarrow H_{\acute et}^{m-1}(Z',\mathcal F)\longrightarrow H_{\acute et}^m(X,\mathcal F)\longrightarrow H_{\acute et}^m(X',\mathcal F)\oplus H_{\acute et}^m(Z,\mathcal F)\longrightarrow H_{\acute et}^m(Z',\mathcal F)\longrightarrow\cdots.$$
Then, the canonical map
$$H_{\acute et}^m(X,\mathcal F)\longrightarrow H_{eh}^m(X,\mathcal F_{eh})$$
is an isomorphism for all $m.$ A similar statement holds for the Nisnevich and cdh topologies as well.   
\end{thm}

\begin{proof}
We prove the theorem for the eh topology. The proof for the cdh topology is verbatim.

Let $C^\bullet:=cone(\mathcal F\longrightarrow Ro_*\mathcal F_{eh}),$ where $o$ is the forgetful functor. We need to show that $H_{\acute et}^i(X,C^\bullet)$ vanishes for all smooth schemes $X$ and $i\in \Z.$ Suppose it is not true and there is some $X$ and $i$ for which the cohomology group $H_{\acute et}^i(X,C^\bullet)$ is non-trivial. First choose the smallest such integer $i$ when $X$ varies all smooth schemes, and then choose $X$ of the smallest dimension for which $H_{\acute et}^i(X,C^\bullet)$ is non-trivial. Let $u$ be a non-zero element of this group.

It follows from the definition that the eh sheafification of $C^\bullet$ is acyclic. Hence, the eh sheafification of the presheaf that sends $U\in Sm/\C$ to $H_{\acute et}^i(U,C^\bullet)$ is trivial. This means that there is some eh cover $T\longrightarrow X$ such that the image of $u$ under the homomorphism $H_{\acute et}^i(X,C^\bullet)\longrightarrow H_{\acute et}^i(T,C^\bullet)$ vanishes. 

The eh cover $T\longrightarrow X$ has a refinement (\cite[Proposition 2.3]{Geisser eh})
\begin{displaymath}
\xymatrix{ U \ar[r]^f \ar[d] & X' \ar[d]^b\\
T \ar[r] & X,}
\end{displaymath}
where $b$ is a series of blow-ups along smooth centers and $f$ is an \'etale cover. We claim that the image of $u$ in $H_{\acute et}^i(X',C^\bullet)$ is non-zero. We may obviously assume that the map $b$ consists of a single blow-up, say
\begin{displaymath}
\xymatrix{ Z'\ar[r]^{i'} \ar[d]_{b'} & X' \ar[d]^b\\
Z\ar[r]_i & X.}
\end{displaymath}
Let $\mathcal F\longrightarrow\mathcal I^\bullet$ and $\mathcal F_{eh}\longrightarrow \mathcal J^\bullet$ be injective resolutions in the respective topoi. Then, since $Ro_*\mathcal F_{eh}=\mathcal J^\bullet$ is still a complex of injective sheaves in the \'etale topology, the canonical map $C^\bullet= cone(\mathcal F\longrightarrow \mathcal J^\bullet) \longrightarrow cone(\mathcal I^\bullet\longrightarrow \mathcal J^\bullet)=:\mathcal K^\bullet$ is an injective resolution of $C^\bullet.$

The abstract blow-up triangle for eh cohomology and the \'etale cohomology of $\mathcal F$ and the change of sites give the distinguished triangles in the first and the second rows in the diagram:
\begin{displaymath}
\xymatrix{\mathcal I^\bullet(X) \ar[r] \ar[d] & \mathcal I^\bullet(X')\oplus\mathcal I^\bullet(Z) \ar[r] \ar[d] & \mathcal I^\bullet(Z')\ar[r]^-{[+1]} \ar[d] &\\
\mathcal J^\bullet(X) \ar[r] \ar[d] & \mathcal J^\bullet(X')\oplus\mathcal J^\bullet(Z) \ar[r] \ar[d] & \mathcal J^\bullet(Z')\ar[r]^--{[+1]} \ar[d] &\\
\mathcal K^\bullet(X) \ar[r] \ar[d]^-{[+1]} & \mathcal K^\bullet(X')\oplus\mathcal K^\bullet(Z) \ar[r] \ar[d]^-{[+1]} & \mathcal K^\bullet(Z')\ar[d]^-{[+1]} \ar[r]^-{[+1]} & \\
&&}
\end{displaymath}
By the octahedral axiom, the bottom row is also a distinguished triangle. Since the global section of $\mathcal K^\bullet$ calculates the \'etale hypercohomology of $C^\bullet,$ we obtain the long exact sequence
$$\longrightarrow H_{\acute et}^{i-1}(Z',C^\bullet)\longrightarrow H_{\acute et}^i(X,C^\bullet)\longrightarrow H_{\acute et}^i(X',C^\bullet)\oplus H_{\acute et}^i(Z,C^\bullet)\longrightarrow. $$
By our choice of $i$ and $X,$ the cohomology groups of $Z$ and $Z'$ of degree less than or equal to $i$ vanish. Therefore, the map 
$$H_{\acute et}^i(X,C^\bullet)\longrightarrow H_{\acute et}^i(X',C^\bullet)$$
is injective. Hence, the image $u'$ of $u$ in $H_{\acute et}^i(X',C^\bullet)$ is non-zero.

Now, consider the spectral sequence
$$E_1^{p,q}=H_{\acute et}^q(U\times_{X'}\cdots\times_{X'}U,C^\bullet)\Longrightarrow H_{\acute et}^{p+q}(X',C^\bullet).$$
By our choice of $i,$ $E_1^{p,q}\cong H_{\acute et}^q(U\times_{X'}\cdots\times_{X'}U,C^\bullet)=0$ if $q<i.$ Thus, the canonical map $H_{\acute et}^i(X',C^\bullet)\longrightarrow H_{\acute et}^i(U,C^\bullet)$ is injective, which means that the image of $u'$ is non-zero. This is a contradiction because $b\circ f:U\longrightarrow X$ factors through the eh cover $T\longrightarrow X$ on which $u$ vanishes. 
\end{proof}

\begin{cor}\label{cor: eh descent for smooth schemes}
Let $\mathcal F\in \{BD(n)_{Zar}, DR_{Zar}, DR_{Zar}^{\geq n}, DR_{Zar}^{<n}, B(n)_{Zar},B(\Z/a)_{Zar}\}.$ Then, the canonical maps
$$H_{Zar}^m(X,\mathcal F)\longrightarrow H_{cdh}^m(X,\mathcal F_{cdh})\longrightarrow H_{eh}^m(X,\mathcal F_{eh})$$
are isomorphisms for any $X\in Sm/\C.$
\end{cor}

\begin{proof}
By Theorem~\ref{thm: etale descent}, it suffices to observe that Deligne-Beilinson, de Rham, Hodge filtration and Betti cohomologies have abstract blow-up sequences for smooth schemes. All these follow from the abstract blow-up sequence for Betti cohomology, strict compatibility of morphisms of mixed Hodge structures (\cite[Th\'eor\`eme 2.3.5~(iii)]{DeligneII}) and Proposition~\ref{prop: triangles}.
\end{proof}

\begin{cor}\label{cor: eh descent for arbitrary schemes}
For $\tau\in\{cdh, eh\}$ and any $X\in Sch/\C,$ there are canonical isomorphisms
\begin{enumerate}
\item $H_{\tau}^m(X,DR_{\tau})\cong H_{B}^m(X,\C),$
\item $H_{\tau}^m(X,B(n)_{\tau})\cong H_{B}^m(X,\Z(n)),$
\item $H_{\tau}^m(X,B(\Z/a)_{\tau})\cong H_{B}^m(X,\Z/a).$
\end{enumerate}
\end{cor}

\begin{proof}
Let us prove (ii) for $\tau= Nis.$ The other cases are similar. 

Let $X_\bullet\longrightarrow X$ be a smooth proper cdh hypercover. Let $\mathcal I^\bullet$ be an injective resolution of $\Omega^\bullet$ in $Sh_{cl}(Sm/\C).$ Then, as in the proof of Theorem~\ref{thm: etale descent}, we have $DR_{cdh}=\mathcal I_{cdh}^\bullet.$ Thus, we have a sequence of isomorphisms
\begin{eqnarray*}
H_{cdh}^m(X, DR_{cdh}) &\cong& H_{cdh}^m(X_\bullet,DR_{cdh}) \\
&=& H_{cdh}^m(X_\bullet,\mathcal I_{cdh}^\bullet) \\
&\buildrel\text{Cor.~\ref{cor: eh descent for smooth schemes}}\over\cong& H_{Nis}^m(X_\bullet, \mathcal I^\bullet) \\
&\buildrel\text{(a)}\over\cong& H^m(Tot\mathcal I^\bullet(X_\bullet)) \\
&\cong& H_{an}^m(X_\bullet, \Omega^\bullet) \\
&\cong& H_{an}^m(X_\bullet, \underline\C) \\
&\buildrel\text{(b)}\over\cong& H_{B}^m(X, \C).
\end{eqnarray*}
The isomorphism (a) holds because $\mathcal I^\bullet$ is still a complex of injective sheaves when regarded in the Nisnevich topology, and (b) is due to Deligne's cohomological descent.
\end{proof}

With Theorem~\ref{thm: cdh descent for DB}, we can extend Definition~\ref{defn: DB cohomology for smooth schemes}.

\begin{defn}\label{defn: cdh definition of DB}
The {\bf Deligne-Beilinson cohomology (resp., that without compact supports) of an arbitrary scheme $X\in Sch/\C$} is defined as
$$H_{\DB}^m(X,\Z(n)):=H_{eh}^m(X,DB(n)_{eh}),$$
\begin{center}
(resp., $H_{c,\DB}^m(X,\Z(n)):=H_{c,eh}^m(X,DB(n)_{eh})$)
\end{center}
where the eh hypercohomology (resp., with compact supports) on the right hand side is as defined in \cite[Section 3]{Geisser eh}.
\end{defn}

\begin{rem}
The use of eh hypercohomology instead of cdh is simply a matter of choice. Indeed, Corollary~\ref{cor: eh descent for smooth schemes} gives an isomorphism between the spectral sequences 
$$E_{1,cdh}^{p,q}=H_{cdh}^q(X_p,DB(n)_{cdh})\Longrightarrow H_{cdh}^{p+q}(X,DB(n)_{cdh})$$
and
$$E_{1,eh}^{p,q}=H_{eh}^q(X_p,DB(n)_{eh})\Longrightarrow H_{eh}^{p+q}(X,DB(n)_{eh})$$
where $X_\bullet\longrightarrow X$ is a smooth cdh hypercover.

The equivalence of cdh and eh definitions for cohomology with compact supports follows from the proper case because both cdh and eh hypercohomologies with compact supports have localization sequences compatible under the change of topologies.
\end{rem}

Deligne-Beilinson cohomology is related to Hodge filtrations and integral Betti cohomology by long exact sequences.

\begin{prop}\label{prop: reasonable}
For an arbitrary scheme $X$ over $\C,$ there are long exact sequences
\begin{equation}\label{reasonable les}
\cdots\longrightarrow H_{B}^{m-1}(X,\C)/F^n\longrightarrow H_{\DB}^m(X,\Z(n))\longrightarrow H_{B}^m(X,\Z(n))\longrightarrow H_{B}^{m}(X,\C)/F^n\longrightarrow\cdots.
\end{equation}
\begin{equation}\label{reasonable les with compact supports}
\cdots\longrightarrow H_{c,B}^{m-1}(X,\C)/F^n\longrightarrow H_{c,\DB}^m(X,\Z(n))\longrightarrow H_{c,B}^m(X,\Z(n))\longrightarrow H_{c,B}^{m}(X,\C)/F^n\longrightarrow\cdots.
\end{equation}
\end{prop}

\begin{proof}
Taking the eh hypercohomology with or without compact supports of the eh sheafification of the distinguished triangle~(\ref{triangle DB}) in Proposition~\ref{prop: triangles}
\begin{center}
$DR_{eh}^{<n}[-1]\longrightarrow DB(n)_{eh}\longrightarrow B(n)_{eh}\buildrel [+1]\over\longrightarrow$ in $D^+(Sh_{eh}(Sch/\C)),$
\end{center}
we obtain the long exact sequence
$$\cdots\longrightarrow H_{(c),eh}^{m-1}(X,DR_{eh}^{<n})\longrightarrow H_{(c),\DB}^m(X,\Z(n))\longrightarrow H_{(c),cdh}^m(X,B(n)_{eh})\longrightarrow H_{(c),eh}^{m}(X,DR_{eh}^{<n})\longrightarrow\cdots.$$
Since $H_{eh}^m(X,B(n)_{eh})$ is isomorphic to $H_{B}^m(X,\Z(n))$ by Corollary~\ref{cor: eh descent for arbitrary schemes}~(ii), it suffices to show the isomorphisms $H_{(c),eh}^{m}(X,DR_{eh}^{<n})\cong H_{(c),B}^{m}(X,\C)/F^n$ and $H_{c,eh}^m(X,B(n)_{eh})\cong H_{c,B}^m(X,\Z(n))$ for all $m.$ Let $X\hookrightarrow \bar X$ be a compactification with the boundary $Z:=\bar X\setminus X.$

Let us deal with the isomorphism for integral Betti cohomology with compact supports first. Let $B(n)_{eh}\longrightarrow \mathcal I^\bullet$ be an injective resolution in $Sh_{eh}(Sch/\C).$ Then, the isomorphism $\Z^c(X)_{eh}\cong cone(\Z(Z)_{eh}\longrightarrow \Z(\bar X)_{eh})$ of eh sheaves (see \cite[Corollary 3.9]{Friedlander-Voevodsky}) induces the isomorphism 
$$H_{c,eh}^m(X,B(n)_{eh})\cong H^m(cone(\mathcal I^\bullet(\bar X)\longrightarrow \mathcal I^\bullet(Z))[-1]).$$
Since $\mathcal I^\bullet(\bar X)$ and $\mathcal I^\bullet(Z)$ respectively calculate Betti cohomology of $\bar X$ and $Z$ by Corollary~\ref{cor: eh descent for arbitrary schemes}~(ii), the right hand side is isomorphic to $H_{c,B}^m(X,\Z(n))$ by comparing localization sequences.

As for the isomorphism $H_{c,eh}^{m}(X,DR_{eh}^{<n})\cong H_{c,B}^{m}(X,\C)/F^n,$ let us first note that by the same argument as above, we may derive the isomorphism $H_{c,eh}^m(X, DR_{eh})\cong H_{c,B}^m(X,\C)$ for an arbitrary scheme $X$ follows from the case of proper schemes in Corollary~\ref{cor: eh descent for arbitrary schemes}~(i). Now, by the long exact sequence for eh cohomology with compact supports associated with the eh sheafification of the distinguished triangle~(\ref{triangle hodge}), it is enough to show that the canonical map 
$$H_{c,eh}^m(X,DR_{eh}^{\geq n})\buildrel f\over\longrightarrow H_{c,eh}^m(X, DR_{eh})\cong H_{c,B}^m(X,\C)$$
is injective and its image is nothing but the $n$-th Hodge filtration $F^n$ on $H_{c,sing}^m(X,\C).$ Both of these follow from the following commutative diagram below, in which we set: $i: Z:=\bar X\setminus X\hookrightarrow \bar X$ is the inclusion of the reduced closed subscheme $Z,$ $\epsilon: \bar X_\bullet\longrightarrow \bar X$ and $\eta:Z_\bullet\longrightarrow Z$ are smooth proper cdh hypercovers of, respectively, $\bar X$ and $Z$ with a commutative diagram (see \cite[(6.2.8)]{Deligne})
\begin{displaymath}
\xymatrix{ Z_\bullet \ar[r] \ar[d]_{\eta} & \bar X^\bullet \ar[d]^\epsilon \\
Z \ar[r]_i & \bar X}
\end{displaymath}
and $DR_{eh}^{\geq n}\longrightarrow \mathcal J^\bullet$ and $DR_{eh}\longrightarrow \mathcal K^\bullet$ are injective resolutions in $Sh_{eh}(Sch/\C).$ We shall also write the Godement (injective) resolution of an analytic sheaf $\mathcal F$ as $Gd^\bullet\mathcal F.$
\begin{equation}\label{bigabove}
\xymatrix{ H_{c,eh}^m(X,DR_{eh}^{\geq n}) \ar[r]^-f \ar[d]^-\cong & H_{c,eh}^m(X,DR_{eh}) \ar[d]_-\cong\\
H^m(cone(\mathcal J^\bullet(\bar X)\longrightarrow \mathcal J^\bullet(Z))[-1]) \ar[r] \ar[d]_-{cone(\epsilon^*,\eta^*)}^-\cong & H^m(cone(\mathcal K^\bullet(\bar X)\longrightarrow \mathcal K^\bullet(Z))[-1]) \ar[d]^-{cone(\epsilon^*,\eta^*)}_-\cong \\
H^m(cone(Tot\mathcal J^\bullet(\bar X_\bullet)\longrightarrow Tot\mathcal J^\bullet(Z_\bullet))[-1])  \ar[r] & H^m(cone(Tot\mathcal K^\bullet(\bar X_\bullet)\longrightarrow Tot\mathcal K^\bullet(Z_\bullet))[-1]) \\
H^m(cone(TotGd^\bullet\Omega_{\bar X^\bullet}^{\bullet\geq n}(\bar X_\bullet) \longrightarrow  TotGd^\bullet\Omega_{Z_\bullet}^{\bullet\geq n}(Z_\bullet) )) \ar[u]^-a_-\cong \ar[r] &  H^m(cone( TotGd^\bullet\Omega_{\bar X^\bullet}^{\bullet}(\bar X_\bullet) \longrightarrow  TotGd^\bullet\Omega_{Z_\bullet}^{\bullet}(Z_\bullet)  )[-1]) \ar[u]_-b^-\cong \\
H_{an}^m(\bar X, cone(R\epsilon_*\Omega_{\bar X_\bullet}^{\bullet\geq n}\longrightarrow i_*R\eta_*\Omega_{\bar Z_\bullet}^{\bullet\geq n})[-1]) \ar[r]_-g \ar[u]_-\cong & H_{an}^m(\bar X, cone(R\epsilon_*\Omega_{\bar X_\bullet}^\bullet\longrightarrow i_*R\eta_*\Omega_{\bar Z_\bullet}^\bullet)[-1]) \ar[u]^-\cong}
\end{equation}

The maps $a$ and $b$ are the canonical ones induced by the map~(\ref{presheaf cohomology to sheaf cohomology}). (Recall that $DR_{eh}^{\geq n}$ (resp., $DR_{eh}$) is the eh sheafification of the complex of presheaves $Sm/\C\ni U\mapsto\colim_{U\buildrel\text{good}\over\hookrightarrow\bar U}TotGd^\bullet\Omega_{\bar U}^{\bullet\geq n}(\log\bar U\setminus U)(\bar U)$ (resp., $U\mapsto Gd^\bullet\Omega^\bullet(U)$).) By the eh descent proved in Corollary~\ref{cor: eh descent for smooth schemes}, the maps $a$ and $b$ are isomorphisms. Now, the injectivity of $f$ and the desired property of its image follow because the map $g$ is injective by \cite[Scholie 8.1.9~(v)]{Deligne} and its image in $H_{c,cdh}^m(X,DR_{cdh})\cong H_{c,B}^m(X,\C)$ coincides with $F^nH_{c,B}^m(X,\C)$ by definition. This finishes the proof for (\ref{reasonable les with compact supports}).

For (\ref{reasonable les}), it suffices to show the isomorphism $H_{eh}^{m}(X,DR_{eh}^{<n})\cong H_{B}^{m}(X,\C)/F^n.$ By the same argument as above with the triangle~(\ref{triangle hodge}), it suffices to show that the canonical map
$$H_{eh}^m(X,DR_{eh}^{\geq n})\buildrel h\over\longrightarrow H_{eh}^m(X, DR_{eh})\cong H_{B}^m(X,\C)$$
is injective and its image agrees with the $n$-th Hodge filtration $F^n$ on $H_{B}^m(X,\C).$ Choose a commutative diagram 
\begin{displaymath}
\xymatrix{ X_\bullet \ar[r] \ar[d]_\zeta & \bar X_\bullet \ar[d]^{\bar\zeta}\\
X \ar[r] & \bar X}
\end{displaymath}
with smooth proper cdh hypercovers $\zeta$ and $\bar\zeta$ such that $X_i\longrightarrow \bar X_i$ is a good compactification for each $i.$ (For the existence of such hypercovers, see \cite[(6.2.8)]{Deligne} and also the proof of \cite[Theorem 4.7]{Conrad} for details.) Now, we have a commutative diagram
\begin{displaymath}
\xymatrix{ H_{eh}^m(X,DR_{eh}^{\geq n}) \ar[r]^-h \ar[d]^-\cong & H_{eh}^m(X,DR_{eh}) \ar[d]^-\cong\\
H^m(\mathcal J^\bullet(X)) \ar[d]_-{\zeta^*}^-\cong \ar[r] & H^m(\mathcal K^\bullet(X)) \ar[d]_-{\zeta^*}^-\cong  \\
H^m(Tot\mathcal J^\bullet(X_\bullet)) \ar[r]  & H^m(Tot\mathcal K^\bullet(X_\bullet)) \\
H^m(TotGd^\bullet\Omega_{\bar X}^{\bullet\geq n}(\log\bar X_\bullet\setminus X_\bullet)(\bar X_\bullet)) \ar[r] \ar[u]_-\cong^-c  &  H^m(TotGd^\bullet\Omega_{\bar X}^\bullet(\log\bar X_\bullet\setminus X_\bullet)(\bar X_\bullet)) \ar[u]_-\cong^-d\\
H_{an}^m(\bar X,R\bar\zeta_*\Omega_{\bar X_\bullet}^{\bullet\geq n}(\log \bar X_\bullet\setminus X_\bullet)) \ar[u]_-\cong \ar[r] & H_{an}^m(\bar X, R\bar\zeta_*\Omega_{\bar X_\bullet}^\bullet(\log \bar X_\bullet\setminus X_\bullet)) \ar[u]_-\cong\\}
\end{displaymath}
where $c$ and $d$ are the canonical maps as $a$ and $b$ in the diagram~(\ref{bigabove}) and they are isomorphisms by the proof of Corollary~\ref{cor: eh descent for smooth schemes}. The bottom map is injective by \cite[Scholie 8.1.9~(v)]{Deligne} and its image is, by definition, the $n$-th Hodge filtration of $H_{B}^m(X,\C).$
\end{proof}

\begin{prop}\label{prop: DB an exact sequence}
For any scheme $X\in Sch/\C$ and positive integer $a,$ there is a long exact sequence
$$\longrightarrow H_{(c),\DB}^{m}(X,\Z(n))\buildrel \times a\over\longrightarrow H_{(c),\DB}^m(X,\Z(n))\longrightarrow H_{(c),B}^m(X,\Z/a)\longrightarrow H_{(c),\DB}^{m+1}(X,\Z(n))\longrightarrow.$$
\end{prop}

\begin{proof}
Take eh hypercohomology (with compact supports) of the distinguished triangle 
$$DB(n)_{eh}\buildrel\times a\over\longrightarrow BD(n)_{eh}\longrightarrow B(\Z/a)_{eh}\longrightarrow$$
obtained by eh sheafification of (\ref{triangle DB multiplication}) and derive the isomorphism $H_{c,eh}^m(X,B(\Z/a)_{eh})\cong H_{c,B}^m(X,\Z/a)$ from the proper case in Corollary~\ref{cor: eh descent for arbitrary schemes}~(iii) by arguing with localization sequences as in the proof of Proposition~\ref{prop: reasonable}. 
\end{proof}


\section{Intermediate Jacobians}\label{section: Intermediate Jacobians with compact supports}

For a smooth proper connected scheme $X$ over $\C,$ the $r$-th Griffiths intermediate Jacobian is defined as the complex torus 
$$J_G^{r}(X):=H^{2r-1}_{B}(X,\C)/(F^rH^{2r-1}_{B}(X,\C)+H^{2r-1}_{B}(X,\Z(r))),$$
where $F^r$ signifies the $r$-th Hodge filtration. It sits in the short exact sequence~\cite[(7.9)]{Esnault-Viehweg}
\begin{equation}\label{ses: Jac DB Hodge}
0\longrightarrow J_G^{r}(X)\longrightarrow H_\DB^{2r}(X,\Z(r))\longrightarrow Hdg^r(X)\longrightarrow0,
\end{equation}
where the group of Hodge $r$-cycles $Hdg^r(X)$ is by definition $H_{B}^{2r}(X,\Z(r))\cap H^r(X,\Omega_X^r)\subset H_{B}^{2r}(X,\C).$

In this section, we show that this short exact sequence generalizes to an arbitrary scheme $X$ if we use Carlson's $n$-th intermediate Jacobian associated with mixed Hodge structures and our Deligne-Beilinson cohomology. 

\begin{defn}[{A special case of \cite{Carlson}}]\label{defn: intermediate Jacobians with compact supports}
Let $X$ be a scheme over $\C$ and $(m,n)\in \Z\times\Z.$ The $(m,n)$-th intermediate Jacobian (with compact supports) of $X$ is defined as 
$J_{(c)}^{m,n}(X):= H^{m-1}_{(c),B}(X,\C)/(F^nH^{m-1}_{(c),B}(X,\C)+H^{m-1}_{(c),B}(X,\Z(n))),$
\end{defn}

Let us suppose $X$ is connected for the moment. While we have by definition $J_c^{2r,r}(X)=J_c^{2r,r}(X)=J_G^r(X)$ for smooth proper $X,$ an $(m,n)$-th intermediate Jacobian (with compact supports) is not, in general, compact as a complex Lie group. However, by \cite[Lemma 6]{Carlson}, $J_c^{m,n}(X)$ (resp., $J^{m,n}(X)$) is still a generalized torus, i.e. a quotient of a complex vector space by a discrete subgroup, if $m\leq 2n$ (resp., $min(2m-1,2\dim X+1)\leq 2n$) because the highest possible weight of $H^{m-1}_{c,B}(X,\C)$ (resp., $H^{m-1}_{B}(X,\C)$) is $m-1$ (resp., $min\{2m-2, 2\dim X\}$) by \cite[Th\'eor\`eme 8.2.4]{Deligne}.

\begin{prop}\label{prop: Jac DB Hodge sequence}
Let $i:H_{(c),B}^m(X,\Z(n))\longrightarrow H_{(c),B}^m(X,\C)$ be the map induced by the inclusion $\Z(n)\hookrightarrow \C.$ Then, for any scheme $X$ over $\C,$ there are short exact sequences
$$0\longrightarrow J_{c}^{m,n}(X) \longrightarrow H_{c,\DB}^{m}(X,\Z(n)) \longrightarrow i^{-1}(F^nH_{c,B}^m(X,\C))\longrightarrow 0,$$
and 
$$0\longrightarrow J^{m,n}(X) \longrightarrow H_{\DB}^{m}(X,\Z(n)) \longrightarrow i^{-1}(F^nH_{B}^m(X,\C))\longrightarrow 0.$$

Also, there are equalities 
\begin{equation}\label{hodge cycles and torsions}
i^{-1}(F^nH_{c,B}^m(X,\C))=H_{c,B}^m(X,\Z(n))_{tor}
\end{equation}
if $m<2n,$ and
\begin{equation}\label{hodge cycles and torsions without supports}
i^{-1}(F^nH_{B}^m(X,\C))=H_{B}^m(X,\Z(n))_{tor}
\end{equation}
if $h:=min(2m,2\dim X)<2n.$
\end{prop}

\begin{proof}
By Proposition~\ref{prop: reasonable}, we have exact sequences
\begin{equation}\label{jacobian as ker}
0\longrightarrow J_{(c)}^{m,n}(X)\longrightarrow H_{(c),\DB}^m(X,\Z(n))\longrightarrow \mathrm{ker}\{ H_{(c),B}^{m}(X,\Z(n))\buildrel f\over\longrightarrow H_{(c),B}^{m}(X,\C)/F^n\}\longrightarrow 0,
\end{equation}
but $\ker f$ is equal to $i^{-1}(F^nH_{(c),B}^m(X,\C)).$

The equalities (\ref{hodge cycles and torsions}) and (\ref{hodge cycles and torsions without supports}) hold because the maximal weight of $H_{c,B}^m(X,\C)$ (resp., $H_{B}^m(X,\C)$) is $m$ (resp., $h$).
\end{proof}


\section{Deligne-Beilinson cycle maps}\label{section: Deligne-Beilinson cycle maps}

In this section, we construct a cycle map from Suslin-Friedlander's motivic complex $\Z(n)^{SF}$ to the Deligne-Beilinson complex $DB(n)_{Zar}.$ Taking the eh sheafification, we obtain Deligne-Beilinson cycle maps for Lichtenbaum cohomology with and without compact supports of arbitrary schemes over $\C.$ By Lichtenbaum cohomology (resp., with compact supports) for singular schemes, we mean the eh-hypercohomology (resp., with compact supports) of the Suslin-Friedlander motivic complex. We follow Bloch's method in \cite{Bloch cycle map} but carries out his construction at the level of sheaves; this method has appeared, for example, in \cite{Geisser-Levine}. We work in the derived category of Zariski sheaves on the big Zariski site $Sm/\C_{Zar}.$

Recall that Suslin-Friedlander's motivic complex $\Z(n)^{SF}:= C_\bullet z_{equi}(\A^n,0)[-2n]$ is a complex of Zariski (even \'etale) sheaves quasi-isomorphic to Voevodsky's motivic complex on the big Zariski site over any perfect field (\cite[Theorem 16.7]{MVW}). For us, it is important that $\Z(n)^{SF}$ can be regarded as a subcomplex of Bloch's cycle complex. More precisely, the inclusion of cycles is compatible with face maps; thus it induces a chain map
$$i:\Z(n)^{SF}[2n](X)\hookrightarrow z^{n}(X\times \A^n,\bullet)$$
for any scheme $X$ (\cite[Lemma 19.4]{MVW}). The advantage of the Suslin-Friedlander's complex over Bloch's is that the former is defined on a big site whereas the latter is not (it does not have enough contravariant functoriality). Hence, the eh-sheafification process applies only to the former.

Now, suppose $X$ is a smooth scheme over $\C$ and $W$ is a codimension $n$ cycle on $X.$ Choose a smooth compactification $\bar X$ and the closure $\overline{\mathrm{supp}W}$ of $\mathrm{supp}W$ in $\bar X.$ Let $c_\mathcal{D}(\overline{\mathrm{supp}W})\in H_{\DB,\overline{\mathrm{supp}W}}^{2n}(\bar X, \Z(n))$ be the fundamental class defined in \cite[7.1]{Esnault-Viehweg}. The class $c^\DB(W)\in H_{\DB,\mathrm{supp}W}^{2n}(X,\Z(n))$ is defined as the image of $c_\mathcal{D}(\bar W)\in H_{\DB,\overline{\mathrm{supp}W}}^{2n}(\bar X,\Z(n))$
under the restriction $H_{\DB,\overline{\mathrm{supp}W}}^{2n}(\bar X,\Z(n))\longrightarrow H_{\DB,\mathrm{supp}W}^{2n}(X,\Z(n))$ (see [loc. cit., Remark 7.2]).

\begin{lem}\label{lem: fundamental class}
The class $c^\DB(W)\in H_{\DB,\mathrm{supp}W}^{2n}(X,\Z(n))$ does not depend on the compactification $\bar X.$
\end{lem}

\begin{proof}
Suppose $\bar X'$ is another smooth compactification of $X.$ As usual, we may assume that there is a morphism of compactifications $f:\bar X\longrightarrow \bar X'.$ Let $\overline{\mathrm{supp}W'}$ be the closure of $\mathrm{supp}W$ in $\bar X'.$ Then, we have the commutative diagram
\begin{displaymath}
\xymatrix{ H_{\DB,\overline{\mathrm{supp}W}}^{2n}(\bar X, \Z(n)) \ar[dr]^-{\text{res.}} \\
&    H_{\DB,\mathrm{supp}W}^{2n}(X,\Z(n)) \\ 
H_{\DB,\overline{\mathrm{supp}W'}}^{2n}(\bar X', \Z(n)) \ar[uu]^-{f^*} \ar[ur]_-{\text{res.}}}
\end{displaymath}
By \cite[Proposition 7.5]{Esnault-Viehweg}, $f^*$ sends $c_\mathcal{D}(\overline{\mathrm{supp}W'})$ to $c_\mathcal{D}(\overline{\mathrm{supp}W}).$ Therefore, by the commutativity of the diagram, the class $c^\DB(W)$ is independent of the choice of a smooth compactification.
\end{proof}

\begin{lem}\label{lem: functoriality}
Let $f:X\longrightarrow Y$ is a morphism of smooth schemes, and let $W$ be a codimension $n$ cycle on $Y$ such that the pullback $f^*W$ is also a codimension $n$ cycle on $X.$ Then, we have $c_\DB(\mathrm{supp}f^*W)=f^*(c_\DB(\mathrm{supp}W))$ in $H_{\DB,\mathrm{supp}f^*W}^{2n}(X,\Z(n)).$ 
\end{lem}

\begin{proof}
Choose smooth compactifications $X\hookrightarrow \bar X$ and $Y\hookrightarrow \bar Y$ with a morphism $\bar f$ that makes the diagram
\begin{displaymath}
\xymatrix{ X \ar@{^{(}->}[r] \ar[d]_-f & \bar X \ar[d]^-{\bar f}\\
Y \ar@{^{(}->}[r] & \bar Y}
\end{displaymath}
commutative. By Lemma~\ref{lem: fundamental class}, we may calculate the classes $c_\DB(\mathrm{supp}f^*W)$ and $c_\DB(\mathrm{supp}W)$ with $\bar X$ and $\bar Y.$ Therefore, since $c_\mathcal{D}(\overline{\mathrm{supp}f^*W})=c_\mathcal{D}(\mathrm{supp}\bar f^*\bar W) =\bar f^*(c_\mathcal{D}(\overline{\mathrm{supp}W}))$ in $H_{\DB,\overline{\mathrm{supp}f^*W}}^{2n}(\bar X,\Z(n))$ by \cite[Proposition 7.5]{Esnault-Viehweg}, the lemma follows.
\end{proof}

We have the following diagram, in which: (a) $tr_{\leq}$ means a good truncation of a complex, (b) For a sheaf $\mathcal F$ on a scheme $S$ with a closed subscheme $T\hookrightarrow S,$ we write $\Gamma_T(S,\mathcal F):=cone(\mathcal F(S)\longrightarrow \mathcal F(S\setminus T))[-1],$ and (c) $DB(n)_{Zar}\longrightarrow \mathcal K^\bullet$ is an injective resolution in $Sh_{Zar}(Sm/\C).$
\begin{equation}\label{diagram: maps of sections}
\xymatrix{ C_r z_{equi}(\A^n,0)(X)   \ar[r]^-{c^\DB} & \colim_{W\in C_r z_{equi}(\A^n,0)(X)}H_{\DB,\mathrm{supp}W}^{2n}(X\times\A^n\times\Delta^r,\Z(n)) \ar[d]_-\cong^-{\text{Theorem~\ref{thm: etale descent}(i)}} \\
& \colim_{W}H_{Zar,\mathrm{supp}W}^{2n}(X\times\A^n\times\Delta^r,DB(n)_{Zar}) \\
& \colim_W tr_{\leq 2n}\Gamma_{\mathrm{supp}W}(X\times\A^n\times\Delta^r,\mathcal K^\bullet)[2n] \ar[d]^b \ar[u]_a^-{\text{qis}} \\
& \Gamma(X\times\A^n\times\Delta^r,\mathcal K^\bullet)[2n],}
\end{equation}
where the map $a$ is a quasi-isomorphism by the weak purity of Deligne-Beilinson cohomology, and $b$ is the map that forgets the truncation and supports.

Since we are only dealing with equidimensional cycles, all maps in the diagram~(\ref{diagram: maps of sections}) are compatible with pullbacks along face maps and contravariant in $X$ with respect to all morphisms by Lemma~\ref{lem: functoriality}. Therefore, we obtain the maps of complexes of presheaves on $Sm/\C:$
\begin{equation}\label{diagram: SF to DB}
\xymatrix{\Z(n)^{SF}[2n]\buildrel\text{def.}\over=C_\bullet z_{equi}(\A^n,0)(-) \ar[r]^-{c^\DB} & \colim_{W\in C_\bullet z_{equi}(\A^n,0)(-)}H_{Zar,\mathrm{supp}W}^{2n}(-\times\A^n\times\Delta^\bullet,DB(n)_{Zar}) \\
& Tot^\oplus \colim_W tr_{\leq 2n}\Gamma_{\mathrm{supp}W}(-\times\A^n\times\Delta^\bullet,\mathcal K^\bullet)[2n] \ar[d]^b \ar[u]_a^-{\text{qis}} \\
& Tot^\oplus \Gamma(-\times\A^n\times\Delta^\bullet,\mathcal K^\bullet)[2n].}
\end{equation}
Furthermore, the homotopy invariance of the Deligne-Beilinson cohomology implies that the following two maps are quasi-isomorphisms:
\begin{equation}\label{diagram: HI of DB}
Tot C^{r,s}:=Tot_{r,s}^\oplus \Gamma(U\times\A^n\times\Delta^r,\mathcal K^s)   \buildrel{\text{qis}}\over\longleftarrow  \Gamma(U\times\A^n,\mathcal K^\bullet) \buildrel{\text{qis}}\over\longleftarrow  \Gamma(U,\mathcal K^\bullet);
\end{equation}
here the first arrow is the inclusion of $\Gamma(U\times\A^n,\mathcal K^s)$ to the $(0,s)$-th direct summand $\Gamma(U\times\A^n\times\Delta^0,\mathcal K^s)$ of the total complex and the second is induced by the projection $U\times \A^n\longrightarrow U.$ Since these two quasi-isomorphisms are contravariant in $U,$ they give maps of complexes of Zariski sheaves on $Sm/\C.$ 

Combining the diagrams~(\ref{diagram: SF to DB}) and~(\ref{diagram: HI of DB}), we obtain the maps of presheaves on $Sm/\C$
\begin{equation}
\xymatrix{\Z(n)^{SF} \ar[r]^-{c^\DB} &   \colim_{W\in  C_\bullet z_{equi}(\A^n,0)(X)}H_{Zar,\mathrm{supp}W}^{2n}(-\times\A^n\times\Delta^\bullet,DB(n)_{Zar})[-2n]\\
& Tot^\oplus \colim_W tr_{\leq 2n}\Gamma_{\mathrm{supp}W}(-\times\A^n\times\Delta^\bullet,\mathcal K^\bullet) \ar[d]^-b \ar[u]_-a^-{\text{qis}} \\
& Tot^\oplus \Gamma(-\times\A^n\times\Delta^\bullet,\mathcal K^\bullet) \\
& \mathcal K^\bullet \ar[u]^{\text{qis}} & DB(n)_{Zar} \ar[l]_-{\text{qis as Zariski sheaves}}}
\end{equation}
Taking the Zariski sheafification, we obtain the corresponding diagram of complexes of Zariski sheaves on $Sm/\C$. Let us record this as a theorem.

\begin{thm}\label{thm: DB cycle map}
There is a map $$c^\DB(n):\Z(n)^{SF}\longrightarrow DB(n)_{Zar}$$ in $D(Sh_{Zar}(Sm/\C))$ such that the induced map of Zariski hypercohomology of any smooth scheme $X$
$$H_{Zar}^m(X,\Z(n)^{SF})\longrightarrow H_{Zar}^m(X,DB(n)_{Zar})\cong H_\DB^m(X,\Z(n))$$
agrees with the Deligne-Beilinson cycle map in \cite{Bloch cycle map} via the canonical isomorphism $H_{Zar}^m(X,\Z(n)^{SF})\cong CH^n(X,2n-m)$ constructed in \cite[Chapter 19]{MVW}.
\end{thm}

\begin{proof}
This follows from the construction of $c^\DB(n)$ and \cite[Theorem 19.8, Proposition 19.12]{MVW}. 
\end{proof}

By eh sheafification, $c^\DB(n)$ induces the map
$$c_{eh}^\DB(n):\Z(n)_{eh}^{SF}\longrightarrow DB(n)_{eh}^\DB$$
in $D(Sh_{eh}(Sch/\C)).$

\begin{defn}\label{defn: DB cycle map with compact supports}
Let $X$ be an arbitrary scheme over $\C$. The {\bf Deligne-Beilinson cycle map for Lichtenbaum cohomology (resp., with compact supports)}
$$cl_{L}^\DB: H_{L}^m(X,\Z(n))\longrightarrow H_{\DB}^m(X,\Z(n))$$
\begin{center}
(resp., $cl_{c,L}^\DB: H_{c,L}^m(X,\Z(n))\longrightarrow H_{c,\DB}^m(X,\Z(n))$)
\end{center}
is the map induced by taking the eh hypercohomology (resp., with compact supports) of $c_{eh}^\DB(n).$
\end{defn}

The same construction applied to the Betti cycle class---note that Betti cohomology also has weak purity and $\A^1$-homotopy invariance---gives the morphism
$$c^B(n): \Z(n)^{SF}\longrightarrow B(n)_{Zar}$$
in $D(Sh_{Zar}(Sm/\C)).$ Taking the eh hypercohomology, we define Betti cycle maps for Lichtenbaum cohomology (with compact supports)
$$cl_{(c),L}^B: H_{(c),L}^m(X,\Z(n))\longrightarrow H_{(c),B}^m(X,\Z(n))$$
for an arbitrary $X\in Sch/\C.$

By construction and the definition of the Deligne-Beilinson cycle map (it is defined by lifting Betti fundamental classes; see \cite[(7.1)]{Esnault-Viehweg}), there is a commutative diagram
\begin{equation}\label{DB and B cycle compatible}
\xymatrix{ \Z(n)^{SF} \ar[r]^{c^\DB}(n) \ar@{=}[d] & BD(n)_{Zar} \ar[d] \\
\Z(n)^{SF} \ar[r]^{c^B}(n)  & B(n)_{Zar}}
\end{equation}
where the right vertical arrow is the map induced by the projection $\{DB(n)_{\bar T,T}\}\longrightarrow \{B(n)_{\bar T,T}\}$ in  $Sh_{an}(\Pi).$

With Proposition~\ref{prop: Jac DB Hodge sequence} and the eh sheafification of the diagram~(\ref{DB and B cycle compatible}), the Deligne-Beilinson cycle maps restrict to the homological part $H_{(c),L,hom}^m(X,\Z(m)):=\ker\{H_{(c),L}^m(X,\Z(m))\buildrel cl_{(c),L}^B\over\longrightarrow H_{(c),B}^m(X,\Z(n))\}$ as indicated in the following diagram:
\begin{equation}\label{for corollary}
\xymatrix{ 0\ar[r] & H_{(c),L,hom}^m(X,\Z(n)) \ar[r] \ar[d]_-{AJ_{(c),L}^{m,n}} & H_{(c),L}^m(X,\Z(n)) \ar[d]_-{cl_{(c),L}^\DB} \ar[r]^-{cl_{(c),L}^B} & im(cl_{(c),L}^B) \ar[r] \ar@{_{(}->}[d]_-{\text{inc.}} & 0 \\
0\ar[r] & J_{(c)}^{m,n}(X) \ar[r] &  H_{(c),\DB}^m(X,\Z(n)) \ar[r] & i^{-1}(F^nH_{(c),B}^m(X,\C)) \ar[r] &0.}
\end{equation}
Here, $AJ_{(c),L}^{m,n}$ is by definition the restriction of $cl_{(c),L}^\DB$ and $i:H_{(c),B}^m(X,\Z(n))\longrightarrow H_{(c),B}^m(X,\C)$ is the map induced by the inclusion $\Z(n)\hookrightarrow \C.$ Let us point out here that $H_{(c),L,hom}^m(X,\Z(m))$ coincides with the subgroup of divisible elements of $H_{(c),L}^m(X,\Z(m))$ as explained in Remark~\ref{rem: max div} below.

\begin{thm}[{Abel-Jacobi and Lefschetz theorems; cf.~\cite[Theorem 7.8]{Arapura}}]\label{thm: AJ and Lefschetz}
For an arbitrary $X\in Sch/\C,$ 
\begin{enumerate}
\item The change of topologies induces an isomorphism $H_{c,cdh}^2(X,\Z(1)_{cdh}^{SF})\longrightarrow H_{c,L}^2(X,\Z(1)).$
\item The Abel-Jacobi map with compact supports
$$AJ_c^{2,1}:  H_{c,L,hom}^2(X,\Z(1))\longrightarrow J_c^{2,1}(X) $$
is an isomorphism.
\item The Betti cycle map with compact supports
$$cl_c^B: H_{c,L}^2(X,\Z(1))\longrightarrow H_{c,B}^2(X,\Z(1))$$
is surjective onto $i^{-1}(F^1H_{c,B}^2(X,\C)).$
\end{enumerate}
\end{thm}

\begin{proof}
For (i), since $\Z(1)^{SF}$ is quasi-isomorphic to $\mathbb G_m[-1]$ on $Sm/\C_{Zar},$ it suffices to show that the canonical map $H_{c,cdh}^1(X,\mathbb G_{m, cdh})\longrightarrow H_{c,eh}^1(X,\mathbb G_{m, eh})$ is an isomorphism. Since both cdh and eh cohomologies with compact supports have localization sequences, it suffices to prove that the change of topologies induces isomorphisms $H_{cdh}^i(S,\mathbb G_{m,cdh})\cong H_{eh}^i(S,\mathbb G_{m,eh})$ for $i=0$ and $1$ for arbitrary proper schemes $S$ over $\C.$

Choose a smooth cdh hypercover $S_\bullet\longrightarrow S$ of $S$ and consider the canonical map of spectral sequences
\begin{displaymath}
\xymatrix{ _{cdh}E_1^{p,q}=H_{cdh}^q(S_p,\mathbb G_{m,cdh}) \ar@{=>}[r] \ar[d] & H_{cdh}^{p+q}(S,\mathbb G_{m,cdh}) \ar[d] \\
_{eh}E_1^{p,q}=H_{eh}^q(S_p,\mathbb G_{m,eh}) \ar@{=>}[r] & H_{eh}^{p+q}(S,\mathbb G_{m,eh})}
\end{displaymath}
Because $S_p$ is smooth, we have $_{cdh}E_1^{p,q}\cong H_{Zar}^q(S_p,\mathbb G_{m})$ and $_{eh}E_1^{p,q}=H_{\acute et}^q(S_p,\mathbb G_{m})$ by \cite[Proposition 13.27]{MVW} (or Theorem~\ref{thm: cdh descent for DB} for the \'etale case), where both isomorphisms are given by change of topologies. Therefore, the canonical maps $_{cdh}E_1^{p,q}\longrightarrow {_{eh}E}_1^{p,q}$ are isomorphisms if $q\leq 1;$ here we used Hilbert's Theorem 90 (\cite[Chapter III, Proposition 4.9]{Milne etale}) for the case $q=1.$ Since the spectral sequences under consideration are in the first quadrant, this is enough to conclude 
$$H_{cdh}^i(S,\mathbb G_{m,cdh})\buildrel\cong\over\longrightarrow H_{eh}^i(S,\mathbb G_{m,eh})$$
for $i=0$ and $1.$

For (ii) and (iii), by the diagram~(\ref{for corollary}), it suffices to show that the Deligne-Beilinson cycle map 
$$cl_{c,L}^\DB: H_{c,L}^2(X,\Z(1))\longrightarrow H_{c,\DB}^2(X,\Z(1))$$
is an isomorphism. The source and the target are both defined by eh hypercohomology, and the cycle map is by definition induced by the morphism $c_{eh}^\DB(1): \Z(1)^{SF}_{eh}\longrightarrow DB(1)_{eh}^\DB$ of complexes of eh sheaves in the derived category. Thus, taking a compactification of $X$ and arguing with localization sequences and the 5-lemma, we can see that it is enough to prove that the cycle map
$$cl_L^\DB: H_L^i(S,\Z(1))\longrightarrow H_{\DB}^i(S,\Z(1))$$
is an isomorphism for any proper scheme $S$ over $\C$ if $i=1$ or $2.$ 

Let $S_\bullet\longrightarrow S$ be a smooth proper cdh hypercover of $S$ and consider the maps of spectral sequences
\begin{displaymath}
\xymatrix{_{M}E_1^{p,q}=H_{cdh}^q(S_p,\mathbb G_{m,cdh}[-1]) \ar@{=>}[r] \ar[d] & H_{cdh}^{p+q}(S,\mathbb G_{m,cdh}[-1]) \ar[d]\\
_{L}E_1^{p,q}=H_{eh}^q(S_p,\mathbb G_{m,eh}[-1]) \ar@{=>}[r] \ar[d] & H_{eh}^{p+q}(S,\mathbb G_{m,eh}[-1]) \ar[d] \\
_{\DB}E_1^{p,q}=H_{\DB}^q(S_p,\Z(1)) \ar@{=>}[r] & H_\DB^{p+q}(S,\Z(1))}
\end{displaymath}
where the top vertical maps are induced by the change of topologies (same as in (i) except for the indexing) and the bottom ones by the composition of $ c_{eh}^\DB(1): \Z(1)^{SF}_{eh}\longrightarrow DB(1)_{eh}^\DB$ with the canonical quasi-isomorphism $\mathbb G_{m,eh}[-1]\simeq\Z(1)_{eh}^{SF}.$

The compositions $_{M}E_1^{p,q}\longrightarrow {_{L}E}_1^{p,q}\longrightarrow {_{\DB}E}_1^{p,q}$ of the $E_1$-terms are isomorphisms for $q\leq2.$ Indeed, it is trivial if $q=0.$ For $q=1$ or $2,$ since $S_p$ is smooth and proper, the claim is equivalent to that Bloch's cycle maps $CH^1(S_p,2-q)\longrightarrow H_\DB^q(S_p,\Z(1))$ are isomorphisms for $q=1$ and $2.$ If $q=1,$ we may assume that $S_p=\Spec\C$ because the structure morphism of any smooth proper connected scheme induces isomorphisms of both higher Chow group $CH^1(-,1)$ and Deligne-Beilinson cohomology $H_\DB^1(-,\Z(1)).$ In this case, the cycle map is indeed an isomorphism $CH^1(\Spec\C,1)\cong\C^*\buildrel\log\over\longrightarrow \C/2\pi i\Z\cong H_\DB^1(\Spec\C,\Z(1))$ by \cite[Section 5.7]{KLMS}. The case for $q=2$ follows from the Abel-Jacobi theorem and the Lefschetz theorem for smooth proper schemes by the diagram~(\ref{picture to extend}).

Now, as we have seen in the proof for (i), the first map $_{M}E_1^{p,q}\longrightarrow {_{L}E}_1^{p,q}$ is an isomorphism if $q\leq2,$ so the second map $_{L}E_1^{p,q}\longrightarrow {_{\DB}E}_1^{p,q}$ is also an isomorphism if $q\leq 2.$ Therefore, we conclude that 
$$cl_L^\DB: H_{L}^i(S,\Z(1))\cong H_{eh}^{i}(S,\mathbb G_{m,eh}[-1])\longrightarrow H_{\DB}^i(S,\Z(1))$$
is an isomorphism for any proper $S$ if $i=1$ or $2.$
\end{proof}


\section{Torsion part of cycle maps}\label{section: Study of torsion elements}

We prove that the cycle map $cl_{c,L}^\DB: H_{c,L}^m(X,\Z(n))\longrightarrow H_{c,\DB}^m(X,\Z(n))$ (resp. $cl_{L}^\DB: H_{L}^m(X,\Z(n))\longrightarrow H_{\DB}^m(X,\Z(n))$) is an isomorphism on torsion for an arbitrary scheme $X$ if $m\leq 2n$ (resp., $min\{2m-1,2\dim X+1\}\leq 2n$). See \cite[Proposition 5.1]{Rosenschon-Srinivas} for the case of smooth projective schemes but with a different construction of a cycle map.

By the commutativity of the diagram~(\ref{DB and B cycle compatible}), we have a map of distinguished triangles for any positive integer $a$
\begin{equation}\label{multiplication diagram}
\xymatrix{\Z(n)^{SF} \ar[r]^-{\times a} \ar[ddd]_-{c^\DB} & \Z(n)^{SF} \ar[ddd]_-{c^\DB} \ar[r]  & \Z(n)^{SF}\otimes\Z/a \ar[r]^-{[+1]}&\\
&& cone(\Z(n)^{SF}\buildrel\times a \over\longrightarrow \Z(n)^{SF}) \ar[d]_-{c^B} \ar[u]_-{\text{canonical isom.}}\\
&& cone(B(n)_{Zar}\buildrel\times a \over\longrightarrow B(n)_{Zar}) \ar[d]^{\text{canonical isom.}}\\
DB(n)_{Zar} \ar[r]^-{\times a} & DB(n)_{Zar} \ar[r]^q & B(\Z/a)_{Zar} \ar[r]^-{[+1]} &}
\end{equation}
with the bottom triangle being the one in (\ref{triangle DB multiplication}).

\begin{lem}\label{lem: mod a isomorphism}
The \'etale sheafification of the composition of the far right vertical arrows in the diagram~(\ref{multiplication diagram})
$$F: \Z(n)^{SF}\otimes\Z/a \longrightarrow B(\Z/a)_{\acute et}$$
is an isomorphism in $D(Sh_{\acute et}(Sm/\C)).$
\end{lem}

\begin{proof}
Let $\mathcal I^\bullet$ be an injective resolution of $\underline{\Z/a}$ in the cl-topology and let $o: Sh_{cl}(Sm/\C)\longrightarrow Sh_{\acute et}(Sm/\C)$ denote the forgetful functor. Then, there is a natural comparison map by Artin between $\Z/a$-coefficient Betti cohomology and $\mu_a^{\otimes n}$-coefficient \'etale cohomology given by a quasi-isomorphism of complexes of \'etale sheaves $A: \mu_a^{\otimes n}\buildrel{\text{qis}}\over\longrightarrow Ro_*(\mathcal I^\bullet);$ see \cite[Chapter III, Lemma 3.15]{Milne etale} and the paragraph that precedes it.

As we have explained at the beginning of the proof of Theorem~\ref{thm: etale descent}, $Ro_*(\mathcal I^\bullet)$ is precisely our $B(\Z/a)_{\acute et}.$ The same construction of a cycle map as in Section~\ref{section: Deligne-Beilinson cycle maps} (but this time with \'etale hypercohomology instead of Zariski hypercohomology) we may lift the \'etale cycle map to the map $c^{\acute et}:\Z(n)^{SF}\longrightarrow \mu_a^{\otimes n}$ in $D(Sh_{\acute et}(Sm/\C)).$ Since Betti and \'etale cycle maps are compatible with Artin's comparison map, we have a commutative diagram
\begin{displaymath}
\xymatrix{ \Z(n)^{SF} \ar[r]^-{c^{\acute et}} \ar[rd]_-{c^B} & \mu_a^{\otimes n} \ar[d]_-A^-{\text{qis}}\\
& B(\Z/a)_{\acute et}.}
\end{displaymath}
This gives rise to 
\begin{displaymath}
\xymatrix{ \Z(n)^{SF}\otimes\Z/a \ar[r]^-{c^{\acute et}} \ar[rd]_-{F} & \mu_a^{\otimes n} \ar[d]_-A^-{\text{qis}}\\
& B(\Z/a)_{\acute et}.}
\end{displaymath}
Since $c^{\acute et}:\Z(n)^{SF}\otimes\Z/a \longrightarrow \mu_a^{\otimes n}$ is a quasi-isomorphism by \cite[Theorem 1.5]{Geisser-Levine}, $F$ is also a quasi-isomorphism.
\end{proof}

\begin{rem}[{Interlude; cf. \cite[Section 3]{Geisser alg rep}}]\label{rem: max div}
For any scheme $X$ over $\C,$ the homological part $H_{(c),L,hom}^m(X,\Z(n))$ is the subgroup of all divisible elements of $H_{(c),L}^m(X,\Z(n)).$ Indeed, the homological part is nothing but the kernel of the canonical map
$$H_{(c),L}^m(X,\Z(n))\longrightarrow \lim_{a\in\Z_{>0}}H_{(c),L}^m(X,\Z(n))/a.$$
This follows from the commutativity of the diagram
\begin{displaymath}
\xymatrixcolsep{3pc}\xymatrix{ H_{(c),L}^m(X,\Z(n)) \ar[r]^-{cl_{(c),L}^B} \ar[d] & H_{(c),B}^m(X,\Z(n))  \ar@{_{(}->}[d]^f\\ 
\lim_{a}H_{(c),L}^m(X,\Z(n))/a \ar@{_{(}->}[d] &  \lim_{a}H_{(c),B}^m(X,\Z(n))/a \ar@{_{(}->}[d] \\
\lim_{a}H_{(c),L}^m(X,\Z/a(n)) \ar[r]^\cong_{\text{Lemma~\ref{lem: mod a isomorphism}}} &  \lim_{a}H_{(c),B}^m(X,\Z/a)}
\end{displaymath}
where $f$ is injective because $H_{(c),B}^m(X,\Z(n))$ is finitely generated.
\end{rem}

From Lemma~\ref{lem: mod a isomorphism} and the diagram~(\ref{multiplication diagram}), we obtain 

\begin{prop}\label{prop: comparison les}
For any $X\in Sch/\C,$ there is a commutative diagram with exact rows of cohomologies with or without compact supports
\begin{displaymath}
\xymatrix{\cdots\ar[r] & H_{(c),L}^{m-1}(X,\Z/a(n)) \ar[r] \ar[d]_{\text{isom.}} & H_{(c),L}^{m}(X,\Z(n)) \ar[r]^-{\times a} \ar[d]_{cl_{(c),L}^\DB} & H_{(c),L}^{m}(X,\Z(n)) \ar[r] \ar[d]_{cl_{(c),L}^\DB} & H_{(c),L}^{m}(X,\Z/a(n)) \ar[r] \ar[d]_{\text{isom.}}& \cdots\\
\cdots\ar[r] & H_{(c),B}^{m-1}(X,\Z/a) \ar[r]  & H_{(c),\DB}^{m}(X,\Z(n)) \ar[r]^-{\times a}  & H_{(c),\DB}^{m}(X,\Z(n)) \ar[r]  & H_{(c),B}^{m}(X,\Z/a) \ar[r] & \cdots}
\end{displaymath}
\end{prop}

\begin{proof}
Take the eh sheafification of the diagram~(\ref{multiplication diagram}) and pass to the eh hypercohomology. The proposition follows from Corollary~\ref{cor: eh descent for smooth schemes}, Definition~\ref{defn: cdh definition of DB} and Lemma~\ref{lem: mod a isomorphism}.
\end{proof}

Let us prove another lemma before the main theorem of this section.

\begin{lem}\label{lem: DB and torsion}
For any scheme $X,$ any positive integer $a$ and $m\leq 2n-1,$ there is a canonical isomorphism
$$H_{c,\DB}^m(X,\Z(n))\otimes\Z/a\buildrel \cong\over\longrightarrow H_{c,B}^{m}(X,\Z(n))_{tor}\otimes\Z/a.$$

Similarly, $H_{\DB}^m(X,\Z(n))\otimes\Z/a\buildrel \cong\over\longrightarrow H_{B}^{m}(X,\Z(n))_{tor}\otimes\Z/a$ if $h:=min\{2m,2\dim X\}\leq 2n-1.$
\end{lem}

\begin{proof}
Consider the long exact sequence 
$$\cdots\longrightarrow H_{(c),\DB}^{m}(X,\Z(n))\buildrel g\over\longrightarrow H_{(c),B}^{m}(X,\Z(n))\buildrel f\over\longrightarrow H_{(c),B}^{m}(X,\C)/F^{n}\longrightarrow\cdots$$
in Proposition~\ref{prop: reasonable}.

We have $\mathrm{im}~g=\ker f=H_{(c),B}^m(X,\Z(n))_{tor}$ in the range under consideration for the weight reason (Proposition~\ref{prop: Jac DB Hodge sequence}). Tensoring $g$ with $\Z/a,$ we obtain the surjection
\begin{displaymath}
\xymatrix{H_{(c),\DB}^{m}(X,\Z(n))\otimes\Z/a\ar@{->>}[r]^-{g_a} & H_{(c),B}^{m}(X,\Z(n))_{tor}\otimes\Z/a.}
\end{displaymath}
This map is also injective because it fits in the commutative diagram
\begin{displaymath}
\xymatrix{  H_{(c),\DB}^{m}(X,\Z(n))\otimes\Z/a \ar@{->>}[r]^-{g_a} \ar@{_{(}->}[d]& H_{(c),B}^{m}(X,\Z(n))_{tor}\otimes\Z/a \ar[d]\\
H_{(c),B}^{m}(X,\Z/a) & H_{(c),B}^{m}(X,\Z(n))\otimes\Z/a \ar[l]}
\end{displaymath}
where all maps are the obvious ones and the left vertical map is injective by Proposition~\ref{prop: DB an exact sequence}.
\end{proof}

\begin{thm}[{cf.~\cite[Proposition 5.1]{Rosenschon-Srinivas}}]\label{thm: torsion main}
Let $X$ be an arbitrary scheme over $\C$ and let $n$ be a positive integer. Then,
\begin{enumerate}
\item $cl_{(c),L}^\DB: H_{(c),L}^m(X,\Z(n)) \longrightarrow H_{(c),\DB}^m(X,\Z(n))$ is surjective on torsion and has a torsion-free cokernel.
\item $cl_{c,L}^\DB: H_{c,L}^m(X,\Z(n)) \longrightarrow H_{c,\DB}^m(X,\Z(n))$ is an isomorphism on torsion if $m\leq 2n.$
\item $cl_L^\DB: H_{L}^m(X,\Z(n)) \longrightarrow H_{\DB}^m(X,\Z(n))$ is an isomorphism on torsion if $min\{2m-1,2\dim X+1\}\leq 2n.$
\end{enumerate}
\end{thm}

\begin{proof}
By Proposition~\ref{prop: comparison les}, we have a diagram with exact rows
\begin{equation}\label{diagram: for torsion part}
\xymatrix{ 0 \ar[r] & H_{(c),L}^{m-1}(X,\Z(n))\otimes \Q/\Z \ar[r] \ar[d]_{cl_{(c),L}^\DB\otimes\Q/\Z} &  H_{(c),L}^{m-1}(X,\Q/\Z(n)) \ar[r] \ar[d]^{\text{isom.}} & H_{(c),L}^m(X,\Z(n))_{tor} \ar[d]_{cl_{(c),L,tor}^\DB} \ar[r] & 0  \\
0 \ar[r] & H_{(c),\DB}^{m-1}(X,\Z(n))\otimes \Q/\Z \ar[r] &  H_{(c),B}^{m-1}(X,\Q/\Z) \ar[r]  & H_{(c),\DB}^m(X,\Z(n))_{tor} \ar[r] & 0}
\end{equation}
Hence, by the snake lemma, $cl_{(c),L}^\DB\otimes \Q/\Z$ is injective and $cl_{(c),L,tor}^\DB$ is surjective. The torsion-freeness of $\mathrm{coker}(cl_{(c)}^\DB)$ follows from these by a diagram chase performed on 
\begin{displaymath}
\xymatrix{ 0 \ar[r] & _{tor}A \ar[r] \ar@{->>}[d]_{cl_{(c),L,tor}^\DB} & A \ar[r] \ar[d]_{cl_{(c),L}^\DB} & A\otimes\Q \ar[r] \ar[d]_{cl_{(c),L}^\DB\otimes \Q} & A\otimes\Q/\Z \ar[r] \ar@{_{(}->}[d]_{cl_{(c),L}^\DB\otimes \Q/\Z} & 0\\
0 \ar[r] & _{tor}B \ar[r] &  B \ar[r] \ar@{->>}[d] & B\otimes\Q \ar[r] \ar@{->>}[d] & B\otimes\Q/\Z \ar[r] & 0\\
&& \mathrm{coker}(cl_{(c),L}^\DB) \ar[r] & \mathrm{coker}(cl_{(c),L}^\DB)\otimes \Q}
\end{displaymath}
where $A=H_{(c),L}^{m}(X,\Z(n))$ and $B=H_{(c),\DB}^{m}(X,\Z(n)).$ This finishes the proof for (i).

For (ii), by the diagram~(\ref{diagram: for torsion part}), it suffices to show that $\mathrm{coker}(cl_{c,L}^\DB\otimes \Q/\Z)$ is trivial. Indeed, by Lemma~\ref{lem: DB and torsion}, we have more strongly
$$H_{c,\DB}^{m-1}(X,\Z(n))\otimes \Q/\Z\cong\colim_a H_{c,\DB}^{m-1}(X,\Z(n))\otimes \Z/a\cong\colim_a H_{c,B}^{m-1}(X,\Z(n))_{tor}\otimes\Z/a =0.$$ 

The proof for (iii) is similar.
\end{proof}

\begin{rem}
Suppose $(m,n)$ is in the range of Theorem~\ref{thm: torsion main}(ii) (resp., Theorem~\ref{thm: torsion main}(iii)). In the proof of the theorem, we showed that the group $H_{c,\DB}^{m-1}(X,\Z(n))\otimes \Q/\Z$ (resp., $H_{\DB}^{m-1}(X,\Z(n))\otimes \Q/\Z$) vanishes. Therefore, $H_{(c),\DB}^m(X,\Z(n))_{tor}$ is canonically isomorphic to $H_{(c),B}^{m-1}(X,\Q/\Z).$
\end{rem}

\begin{cor}\label{cor: on torsion}
In the diagram~(\ref{for corollary}) for any $X\in Sch/\C,$ 
\begin{enumerate}
\item $AJ_c^{m,n}: H_{c,L,hom}^m(X,\Z(n))\longrightarrow J_c^{m,n}(X)$ (resp., $AJ^{m,n}: H_{L,hom}^m(X,\Z(n))\longrightarrow J^{m,n}(X)$) is an isomorphism on torsion if $m\leq 2n$ (resp., $min\{2m-1,2\dim X+1\}\leq 2n$). In these ranges, if $X$ is connected, the Jacobians $J_{(c)}^{m,n}(X)$ are generalized complex tori.
\item The quotient group $i^{-1}(F^nH_{(c),B}^m(X,\C))/im(cl_{(c),L}^B)$ is always torsion free. In particular, $$cl_{c,L}^B: H_{c,L}^m(X,\Z(n))\longrightarrow H_{c,B}^m(X,\Z(n))$$ 
\begin{center}
(resp., $cl_{L}^B: H_{L}^m(X,\Z(n))\longrightarrow H_{B}^m(X,\Z(n))$)
\end{center}
has the image $$i^{-1}(F^nH_{c,B}^m(X,\C))=H_{c,B}^m(X,\Z(n))_{tor}$$ 
\begin{center}
(resp., $i^{-1}(F^nH_{B}^m(X,\C))=H_{B}^m(X,\Z(n))_{tor}$) 
\end{center}
if $m<2n$ (resp., $min\{2m,2\dim X\}<2n$). 
\end{enumerate}
\end{cor}

\begin{proof}
The first part of (i) is immediate from Theorem~\ref{thm: torsion main} (ii) and (iii). The second part of (i) is already explained in the paragraph right after Definition~\ref{defn: intermediate Jacobians with compact supports}.

For (ii), taking the cokernels of the vertical arrows in the diagram~(\ref{for corollary}), we obtain the exact sequence
$$0\longrightarrow \mathrm{coker}~AJ_{(c)}^{m,n}\longrightarrow \mathrm{coker}~cl_{(c),L}^\DB\longrightarrow i^{-1}(F^nH_{(c),B}^m(X,\C))/im(cl_{(c),L}^B)\longrightarrow 0.$$ 
The torsion-freeness of $i^{-1}(F^nH_{(c),B}^m(X,\C))/im(cl_{(c),L}^B)$ follows because $\mathrm{coker}~AJ_{(c)}^{m,n},$ being a quotient of a vector space, is divisible and $\mathrm{coker}~cl_{(c),L}^\DB$ is torsion-free by Theorem~\ref{thm: torsion main} (i).
\end{proof}


\section{Griffiths intermediate Jacobians revisited}\label{section: Griffiths intermediate Jacobians revisited}

To finish this paper, we would like to reconsider Griffiths intermediate Jacobians from the viewpoint of Lichtenbaum cohomology. With Theorem~\ref{thm: torsion main}, we may characterize the ``algebraic part" of Griffiths intermediate Jacobians by a universal property. In the rest of this paper, we assume that $X$ is smooth, proper and connected. 

Samuel \cite{Samuel} introduced regular homomorphisms to relate the algebraic part $CH_{alg}^r(X):=\{z\in CH^r(X)| z~\text{is algebraically equivalent to}~0\}$ of the Chow group $CH^r(X)$ of $X$ over an arbitrary algebraically closed field $k$ with abelian varieties over the same base field $k.$ It is known that, regardless of the base field, there is a universal regular homomorphism if the codimension $r$ is $1$ or $\dim X$ ([ibid.]), but the existence is unknown in other codimensions. However, over the field of complex numbers, we have Abel-Jacobi maps $AJ_{X}^r: CH_{hom}^r(X)\longrightarrow J_G^r(X)$ for all $r,$ and their restrictions to the algebraic part $AJ_{X,alg}^r: CH_{alg}^r(X)\longrightarrow J_G^r(X)_{alg}:=AJ_X^r(CH_{alg}^r(X))$ are known to be regular (see \cite{LiebermanMotive}). $AJ_{X,alg}^r$ agrees with the universal regular homomorphism if $r=1,2$ or $\dim X$ (\cite{Murre}), but it is generally unknown if the algebraic part of Abel-Jacobi maps are universal regular for general $r.$

While the existence of universal regular homomorphisms are not known in general, Geisser (\cite[Section 3]{Geisser alg rep}) showed that an analogue for Lichtenbaum cohomology $H_L^{2r}(X,\Z(r))$ exists for all codimensions $r$ over any algebraically closed base field $k.$ Here, we consider its variant.

Consider the canonical natural maps 
\begin{equation}\label{Chow to etale Chow}
CH^r(X)\cong H^{2r}(X,\Z(r))\buildrel\text{def.}\over=H_{Nis}^{2r}(X,\Z(r)^{SF})\longrightarrow H_{\acute et}^{2r}(X,\Z(r)^{SF})\buildrel\text{def.}\over= H_L^{2r}(X,\Z(r)),
\end{equation}
and define $H_{L,alg}^{2r}(X,\Z(r))$ as the image of $CH_{alg}^r(X)$ under (\ref{Chow to etale Chow}), contrary to the definition in \emph{loc. cit.} Let us say that a group homomorphism
$$\phi: H_{L,alg}^{2r}(X,\Z(r))\longrightarrow A(k)$$
such that $A$ is an abelian variety over $k$ is {\bf $L$-regular} if its composition with the canonical map $CH_{alg}^r(X)\longrightarrow H_{L,alg}^{2r}(X,\Z(r))$ is regular in Samuel's sense, i.e. for any smooth proper connected $T$ over $k$ and any $Y\in CH^r(T\times X),$ the composition
$$T(k)\buildrel w_Y \over\longrightarrow CH_{alg}^r(X)\longrightarrow H_{L,alg}^{2r}(X,\Z(r))\buildrel\phi\over\longrightarrow A(k)$$
is a scheme morphism, where $w_Y$ sends $t\in T(k)$ to the pullback of $Y$ along $X\cong\Spec\C\times X\buildrel t\times id_X\over\longrightarrow T\times X.$

The existence of the universal $L$-regular homomorphism $\Phi^r_{L,X}: H_{L,alg}^{2r}(X,\Z(r))\longrightarrow Alg_{L,X}^r(k)$ for any $r$ and any smooth proper connected scheme $X$ over $k$ follows by the same argument as in ([ibid., Theorem 3.5]). The map $\Phi^r_{L,X}$ is surjective and surjective on torsion by the construction. We have the following corollary to Theorem~\ref{thm: torsion main}, which may be regarded as an algebraic construction of the algebraic part of Griffiths intermediate Jacobians and Abel-Jacobi maps which works over any characteristic.

\begin{cor}
Let $X$ be a smooth proper connected scheme over $\C.$ The composition
$$CH_{alg}^r(X)\longrightarrow  H_{L,alg}^{2r}(X,\Z(r))\buildrel \Phi^r_{L,X}\over\longrightarrow Alg_{L,X}^r(\C)$$
of the universal $L$-regular homomorphism $\Phi^r_{L,X}$ with the canonical map $CH_{alg}^r(X)\longrightarrow  H_{L,alg}^{2r}(X,\Z(r))$ is nothing but the algebraic part of the Abel-Jacobi map $AJ_{X,alg}^r: CH_{alg}^r(X)\longrightarrow J_G^r(X)_{alg}.$ 
\end{cor}

\begin{proof}
By Theorems~\ref{thm: etale descent} (i) and~\ref{thm: DB cycle map}, the Deligne-Beilinson cycle map factors through Lichtenbaum cohomology as
$$CH^r(X)\longrightarrow H_L^{2r}(X,\Z(r))\buildrel cl_L^\DB\over\longrightarrow H_\DB^{2r}(X,\Z(r)).$$
By \cite[Theorem 7.11]{Esnault-Viehweg}, restricting the above maps to the algebraic part, we obtain the factorization of the Abel-Jacobi map $AJ_{X,alg}^r$
\begin{displaymath}
\xymatrix{CH_{alg}^r(X)\ar[rr]^-{AJ_{X,alg}^r} \ar[dr] && J_G^r(X)_{alg}\\
& H_{L,alg}^{2r}(X,\Z(r))\ar[ur]_-{AJ^{2r,r}}}
\end{displaymath}
where $AJ^{2r,r}$ denotes the restriction of $cl^\DB.$

Now, we have the diagram
\begin{displaymath}
\xymatrix{ CH_{alg}^r(X) \ar[r] \ar@/_1pc/@{->>}[drr]_-{AJ_{X,alg}^r} & H_{L,alg}^{2r}(X,\Z(r)) \ar@{->>}[r]^-{\Phi^r_{L,X}} \ar[dr]_-{AJ^{2r,r}} & Alg_{L,X}^r(\C) \ar@{..>}[d]^-{\exists!~h~\text{by the universality}}\\
&& J_G^r(X)_{alg}}
\end{displaymath}
The map $h$ is surjective by the commutativity and also injective on torsion because $AJ^{2r,r}$ is injective on torsion by Theorem~\ref{thm: torsion main}~(ii) and $\Phi^r_{L,X}$ is surjective on torsion by the construction. Since $h$ is induced by a morphism of abelian varieties, it is an isomorphism.
\end{proof}


\end{document}